\documentclass[11pt]{amsart}
\usepackage{amsmath,amssymb,graphicx,xspace}

\newtheorem{thm}{Theorem}[section]
\newtheorem{lem}[thm]{Lemma}
\newtheorem{cor}[thm]{Corollary}
\newtheorem{conj}[thm]{Conjecture}
\newtheorem{prop}[thm]{Proposition}

\theoremstyle{definition}
\newtheorem{exmp}[thm]{Example}
\newtheorem{defn}[thm]{Definition}
\newtheorem{rem}[thm]{Remark}

\newcommand{\bbb}[1]{\ensuremath{\mathbb{#1}}}
\newcommand{\D}{\bbb{D}}
\newcommand{\euc}{\bbb{E}}
\newcommand{\R}{\bbb{R}}
\newcommand{\sph}{\bbb{S}}
\newcommand{\Z}{\bbb{Z}}
\newcommand{\cayley}{\textsc{Cay}}
\newcommand{\cat}{\textsc{CAT}}
\newcommand{\E}{\ensuremath{\mathcal{E}}}
\newcommand{\G}{\ensuremath{\mathcal{G}}}
\newcommand{\pe}{\textsf{PE}\xspace}
\newcommand{\ps}{\textsf{PS}\xspace}
\newcommand{\size}[1]{\ensuremath{\vert #1 \vert}}

\begin{document}

\title[$\triangle$'s, $\square$'s and geodesics]{Triangles, squares
  and geodesics}

\author[R.~Levitt]{Rena Levitt}
\address{Dept. of Math.\\ 
  Pomona College
  Claremont, CA 91711} 
\email{rena.levitt@pomona.edu}

\author[J.~McCammond]{Jon McCammond} 
\address{Dept. of Math.\\ 
  U. C. Santa Barbara\\ 
  Santa Barbara, CA 93106} 
\email{jon.mccammond@math.ucsb.edu}

\date{\today}

\begin{abstract}
  In the early 1990s Steve Gersten and Hamish Short proved that
  compact nonpositively curved triangle complexes have biautomatic
  fundamental groups and that compact nonpositively curved square
  complexes have biautomatic fundamental groups.  In this article we
  report on the extent to which results such as these extend to
  nonpositively curved complexes built out a mixture of triangles and
  squares.  Since both results by Gersten and Short have been
  generalized to higher dimensions, this can be viewed as a first step
  towards unifying Januszkiewicz and {\'S}wi{\.a}tkowski's theory of
  simplicial nonpositive curvature with the theory of nonpositively
  curved cube complexes.
\end{abstract}
\maketitle

\section{Introduction}

Many concepts in geometric group theory, including hyperbolic groups,
$\cat(0)$ groups, and biautomatic groups, were developed to capture
the geometric and computational properties of examples such as
negatively curved closed Riemannian manifolds and closed topological
$3$-manifolds.  While the geometric and computational aspects of
hyperbolic groups are closely interconnected, the relationship between
the geometrically defined class of $\cat(0)$ groups and the
computationally defined class of biautomatic groups is much less
clear.  Examples of biautomatic groups that are not $\cat(0)$ groups
are known and there is an example of a $2$-dimensional piecewise
Euclidean $\cat(0)$ group that is conjectured to be neither automatic
nor biautomatic \cite{El-thesis}.  Nevertheless, several specific
classes of $\cat(0)$ groups are known to be biautomatic, beginning
with the following results by Gersten and Short \cite{GeSh90}.

\begin{thm}[Gersten-Short]\label{thm:gs} 
  Every compact, nonpositively curved triangle complex has a
  biautomatic fundamental group.  Similarly, every compact,
  nonpositively curved square complex has a biautomatic fundamental
  group.
\end{thm}

We conjecture that this result extends to nonpositively curved
complexes built out a mixture of triangles and squares.

\begin{conj}\label{conj:biaut}
  The fundamental group of a compact nonpositively curved
  triangle-square complex is biautomatic.
\end{conj}

As progress towards answering Conjecture~\ref{conj:biaut} in the
affirmative, we establish that for every compact nonpositively curved
triangle-square complex $K$ there exists a canonically defined
language of geodesics that reduces to the regular languages used by
Gersten and Short when $K$ is a triangle complex or a square complex.
On the other hand, by investigating the canonical language of
geodesics within a single flat plane, we highlight one reason why the
``mixed'' case studied here is significantly more difficult than the
``pure'' cases analyzed by Gersten and Short.

The article is structured as follows. The early sections review the
necessary results about piecewise Euclidean complexes, nonpositively
curved spaces, biautomatic groups, and disc diagrams.
Sections~\ref{sec:geos} and~\ref{sec:intervals} contain general
results about combinatorial geodesics in $\cat(0)$ triangle-square
complexes and Section~\ref{sec:gs-geos} constructs the collection of
canonical geodesic paths alluded to above.  Sections~\ref{sec:flats},
\ref{sec:flat-geos} and \ref{sec:periodic-flats} investigate the
behavior of these geodesics in a single triangle-square flat.
Finally, Section~\ref{sec:biaut} outlines the remaining steps needed
to establish Conjecture~\ref{conj:biaut}.

\section{Nonpositive curvature}\label{sec:npc}

We begin by reviewing the theory of nonpositively curved metric spaces
built out of Euclidean polytopes.  Although most of the article
focuses on $2$-dimensional complexes where simplified definitions are
available, the general definitions are given since higher dimensions
occasionally occur.  See \cite{BrHa99} for additional details.

\begin{defn}[Euclidean polytopes]
  A \emph{Euclidean polytope} $P$ is the convex hull of a finite set
  of points in a Euclidean space, or, equivalently, it is a bounded
  intersection of a finite number of closed half-spaces.  A
  \emph{proper face} of $P$ is a nonempty subset of $P$ that lies in
  the boundary of a closed half-space containing $P$.  It turns out
  that every proper face of a polytope is itself a polytope.  There
  are also two trivial faces: the empty face $\emptyset$ and $P$
  itself.  The \emph{interior} of a face $F$ is the collection of its
  points that do not belong to a proper subface, and every polytope is
  a disjoint union of the interiors of its faces.  The
  \emph{dimension} of a face $F$ is the dimension of the smallest
  affine subspace that containing $F$.  A $0$-dimensional face is a
  \emph{vertex} and a $1$-dimensional face is an \emph{edge}. A
  $2$-dimensional polytope is a \emph{polygon}, and an \emph{$n$-gon}
  if it has $n$ vertices.
\end{defn}

\begin{defn}[\pe complexes]
  A \emph{piecewise Euclidean complex} (or \emph{\pe complex}) is the
  regular cell complex that results when a disjoint union of Euclidean
  polytopes are glued together via isometric identifications of their
  faces.  A basic result due to Bridson is that so long as there are
  only finitely many isometry types of polytopes used in the
  construction, the result is a geodesic metric space, i.e. the
  distance between two points in the quotient is well-defined and
  achieved by a path of that length connecting them.  This was a key
  result from Bridson's thesis \cite{Br91} and is the main theorem of
  Chapter I.7 in \cite{BrHa99}.
\end{defn}

The examples under investigation are special types of \pe
$2$-complexes.

\begin{exmp}[Triangle-square complexes]
  A \emph{triangle complex} is a \pe complex where every polytope used
  in its construction is isometric to an equilateral triangle with
  unit length sides.  Note that these complexes are not necessarily
  simplicial since two edges of the same triangle can be identified.
  They agree instead with Hatcher's notion of a $\Delta$-complex
  \cite{Ha02}.  Similarly, a \emph{square complex} is a \pe complex
  where every polytope used is a unit square and a
  \emph{triangle-square complex} is one where every polytope used is
  either a unit equilateral triangle or a unit square.
\end{exmp}

Euclidean polytopes and \pe complexes have spherical analogs that are
needed for our characterization of nonpositive curvature.

\begin{defn}[Spherical polytopes]
  A \emph{spherical polytope} is an intersection of a finite number of
  closed hemispheres in $\sph^n$, or, equivalently, the convex hull of
  a finite set of points in $\sph^n$.  In both cases there is an
  additional requirement that the intersection or convex hull be
  contained in some open hemisphere of $\sph^n$.  This avoids
  antipodal points and the non-uniqueness of geodesics connecting
  them.  With closed hemispheres replacing closed half-spaces and
  lower dimensional unit subspheres replacing affine subspaces, the
  other definitions are unchanged.
\end{defn}

\begin{defn}[\ps complexes]
  A \emph{piecewise spherical complex} (or \emph{\ps complex}) is the
  regular cell complex that results when a disjoint union of spherical
  polytopes are glued together via isometric identifications of their
  faces.  As above, so long as there are only finitely many isometry
  types of polytopes used in the construction, the result is geodesic
  metric space.
\end{defn}

\begin{defn}[Links]
  Let $F$ be a face of a Euclidean polytope $P$.  The \emph{link of
    $F$ in $P$} is the spherical polytope of unit vectors $u$
  perpendicular to the affine hull of $F$ and pointing towards $P$ (in
  the sense that for any point $x$ in the interior of $F$ there is a
  sufficiently small positive value of $\epsilon$ such that
  $x+\epsilon u$ lies in $P$).  For example, if $P$ is a polygon, and
  $F$ is one of its vertices, the link of $F$ in $P$ is a circular arc
  whose length is the radian measure of the interior angle at $F$, and
  if $F$ is an edge of $P$ then the link of $F$ in $P$ is a single
  point representing one of the two unit vectors perpendicular to the
  affine hull of $F$.  Finally, when $K$ is a \pe complex and $F$ is
  one of its cells, the \emph{link of $F$ in $K$} is the \ps complex
  obtained by gluing together the link of $F$ in each of the various
  polytopal cells of $K$ containing it in the obvious way.  When $F$
  is a vertex, its link is a \ps complex isometric (after rescaling)
  with the subspace of points distance $\epsilon$ from $F$.  A vertex
  link in a triangle-square complex is a metric graph in which every
  edge either has length $\frac{\pi}{3}$ or $\frac{\pi}{2}$ depending
  on whether it came from a triangle or a square.
\end{defn}

We turn now to curvature conditions.  A \emph{$\cat(0)$ space} is a
geodesic metric space in which every geodesic triangle is ``thinner''
than its comparison triangle in the Euclidean plane and a space is
\emph{nonpositively curved} if every point has a neighborhood which is
$\cat(0)$.  A key consequence of the $\cat(0)$ condition is that every
$\cat(0)$ space is contractible.  For a \pe complex these conditions
are characterized via geodesic loops in its links.

\begin{defn}[Geodesics and geodesic loops]
  A \emph{geodesic} in a metric space is an isometric embedding of a
  metric interval and a \emph{geodesic loop} is an isometric embedding
  of a metric circle.  A \emph{local geodesic} and \emph{local
    geodesic loop} are weaker notions that only require the image
  curves be locally length minimizing.  For example, a path more than
  halfway along the equator of a $2$-sphere is a local geodesic but
  not a geodesic and a loop that travels around the equator twice is a
  local geodesic loop but not a geodesic loop.  A loop of length less
  than $2\pi$ is called \emph{short}.
\end{defn}

\begin{prop}[Nonpositive curvature]\label{prop:char-npc}
  A \pe complex with only finitely many isometry types of cells fails
  to be nonpositively curved iff it contains a cell whose link
  contains a short local geodesic loop, which is true iff it contains
  a cell whose link contains a short geodesic loop.
\end{prop}

The first equivalence is the traditional form of Gromov's link
condition.  The second follows because Brian Bowditch proved in
\cite{Bo95} that when a link of a \pe complex contains a short local
geodesic loop, then some link (possibly a different one) contains a
short geodesic loop.  See \cite{BrMc-ortho} for a detail exposition of
this implication.  We record one corollary for future use.

\begin{cor}\label{cor:char-npc}
  Let $Y \to X$ be a cellular map between \pe $2$-complexes that is an
  immersion away from the $0$-skeleton of $Y$ and an isometry on the
  interior of each cell.  If $X$ is nonpositively curved then so is
  $Y$.
\end{cor}

\begin{proof}
  First note that in a $2$-complex, the vertex links are the only
  links that can contain short local geodesic loops.  Next, by
  hypothesis, the vertex links of $Y$ immerse into the vertex links of
  $X$ and, as a consequence, if $Y$ contains a vertex $v$ whose link
  contains a short local geodesic loop, then the link of $f(v)$ in $X$
  contains a short local geodesic loop.
  Proposition~\ref{prop:char-npc} completes the proof.
\end{proof}

\begin{prop}[$\cat(0)$]\label{prop:char-cat}
  A \pe complex with only finitely many isometry types of cells is
  $\cat(0)$ iff it is connected, simply connected, and nonpositively
  curved.
\end{prop}

This follows from the Cartan-Hadamard Theorem: the universal cover of
a complete, connected, nonpositively curved metric space is $\cat(0)$.

\section{Biautomaticity}\label{sec:biaut}

Biautomaticity involves distances in Cayley graphs and regular
languages accepted by finite state automata.  The standard reference
is \cite{ECHLPT92}.

\begin{defn}[Paths]
  A \emph{path} in a graph $\Gamma$ can be thought of either as an
  alternating sequence of vertices and incident edges, or as a
  combinatorial map $\alpha:I\to \Gamma$ from a subdivided interval
  $I$.  Every connected graph can be viewed as a metric space by
  making each edge isometric to the unit interval and defining the
  distance between two points to be the length of the shortest path
  connecting them.  In this metric the \emph{length} of a path
  $\alpha$, denoted $\size{\alpha}$, is the number of edges it
  crosses.  Define $\alpha(t)$ to be the start of $\alpha$ when $t$ is
  negative, the point $t$ units from the start of $\alpha$ when $0
  \leq t \leq \size{\alpha}$ and the end of $\alpha$ when $t>
  \size{\alpha}$.
\end{defn}

\begin{defn}[Distance]
  Let $\alpha$ and $\beta$ be two paths in a graph $\Gamma$.  There
  are two natural notions of distance between $\alpha$ and
  $\beta$. The \emph{subspace distance} between $\alpha$ and $\beta$
  is the smallest $k$ such that $\alpha$ (or more specifically its
  image) is contained in a $k$-neighborhood of $\beta$ and $\beta$ is
  contained in a $k$-neighborhood of $\alpha$.  The \emph{path
    distance} between $\alpha$ and $\beta$ is the maximum distance
  between $\alpha(t)$ and $\beta(t)$ as $t$ varies over all real
  numbers.  When the path distance is bounded by $k$, $\alpha$ and
  $\beta$ are said to \emph{synchronously $k$-fellow travel}.
\end{defn}

For later use we note that when geodesics start close together,
bounding their subspace distance also bounds their path distance.

\begin{prop}[Fellow travelling geodesics]\label{prop:ft-geo}
  If $\alpha$ and $\beta$ are geodesic paths in $\Gamma$ that start
  distance $\ell$ apart, and the subspace distance between them is
  $k$, then the path distance between them is at most $2k+\ell$.  In
  particular, they synchronously $(2k+\ell)$-fellow travel.
\end{prop}

\begin{proof}
  Let $\alpha(t)$ be a point on $\alpha$ and let $\beta(t')$ be the
  closest point on $\beta$.  By hypothesis the distance between
  $\alpha(t)$ and $\beta(t')$ is at most $k$.  Moreover, because
  $\alpha$ and $\beta$ are geodesics, $t \leq t' + k + \ell$ and $t'
  \leq t + k + \ell$.  Thus $\size{t-t'} \leq k+\ell$. As a result,
  $\alpha(t)$ is within $k + (k+\ell)$ of $\beta(t)$.
\end{proof}

\begin{defn}[Cayley graphs]
  Let $K$ be a one-vertex cell complex with fundamental group $G$ and
  let $A$ index its oriented edges.  The oriented loops indexed by $A$
  represent a generating set for $G$ that is closed under involution
  and the $1$-skeleton of the universal cover of $K$ is called the
  \emph{right Cayley graph of $G$ with respect to $A$} and denoted
  $\cayley(G,A)$.  Alternatively, if $G$ is a group and $A$ is
  generating set closed under involution, $\cayley(G,A)$ is a graph
  with vertices indexed by the elements of $G$ and an oriented edge
  connecting $v_g$ and $v_{g'}$ labeled by $a$ whenever $g\cdot a =
  g'$.  Actually, the oriented edges come in pairs and we add only one
  unoriented edge for each such pair with a label associated with each
  orientation.  It is the right Cayley graph because the generators
  multiply on the right.  The second definition shows that the
  structure of $\cayley(G,A)$ is independent of $K$.  There is a
  natural label-preserving left $G$-action on $\cayley(G,A)$ defined,
  depending on the definition, either by deck transformations or by
  sending $g \cdot v_h$ to $v_{g\cdot h}$.
\end{defn}

\begin{defn}[Languages] 
  An \emph{alphabet} is a set $A$ whose elements are \emph{letters}.
  A \emph{word} is a finite sequence of letters, and the collection of
  all words is denoted $A^*$.  A \emph{language $L$ over an alphabet
    $A$} is any subset of $A^*$.  When $A$ is identified with the
  oriented edges of a one-vertex cell complex $K$ with fundamental
  group $G$, there are natural bijections between languages over $A$,
  collections of paths in $K$, and collections of paths in the
  universal cover of $K$ that are invariant under the $G$-action.  A
  language $L$ over $A$ \emph{maps onto $G$} if the corresponding
  $G$-invariant set of paths contains at least one path connecting
  every ordered pair of vertices, or equivalently contains a word
  representing every element of $G$.
\end{defn}

Two key properties of languages are used to define biautomatic groups.

\begin{defn}[Regular languages]
  A \emph{finite state automaton} is a finite directed graph with
  decorations.  The edges are labeled by a set $A$, there is a
  specified \emph{start vertex}, and a subset of the vertices
  designated as \emph{accept states}.  Concatenating the labels on the
  edges of any path starting at the start vertex and ending at an
  accept state creates an \emph{accepted word}.  The collection of all
  accepted words is the \emph{language accepted} by the automaton.
  Finally, a language is \emph{regular} if it is accepted by some
  finite state automaton.
\end{defn}

\begin{defn}[Fellow travelling languages]
  Let $G$ be a group generated by a finite set $A$ closed under
  inversion and let $L$ be a language over $A$ that maps onto $G$.
  The language $L$ \emph{fellow travels} if there is a constant $k$
  such that all pairs of paths in $\cayley(G,A)$ that start and end at
  most $1$ unit apart and corresponding to words in $L$ synchronously
  $k$-fellow travel.
\end{defn}

The following characterization of biautomaticity is from
\cite[Lemma~2.5.5]{ECHLPT92}.

\begin{thm}[Biautomaticity]\label{thm:biaut}
  If $G$ is a group generated by a finite set $A$ closed under
  inversion and $L$ is a language over $A$ that maps onto $G$, then
  $L$ is part of a biautomatic structure for $G$ iff $L$ is regular
  and fellow travels.
\end{thm}

Recall that a \emph{geometric} action of a group on a space is one
that is proper, cocompact and by isometries and note that the action
of $G$ on its Cayley graph is geometric, but also free and transitive
on vertices.  By focusing on the paths in the Cayley graph
corresponding to the language $L$ and slightly modifying a few
definitions, Jacek \'{S}wi\c{a}tkowski was able in \cite{Sw06} to
extend the characterization of biautomaticity given in
Theorem~\ref{thm:biaut} to groups merely acting geometrically on a
connected graph.

\begin{defn}[Modified notions]
  Let $G$ be a group acting geometrically on a connected graph $X$,
  and let $\mathcal{P}$ be a $G$-invariant collection of paths in $X$.
  The set $\mathcal{P}$ is said to \emph{map onto $G$} if every path
  start and ends in the $G$-orbit of a particular vertex $v$ and every
  ordered pair of vertices in the $G$-orbit of $v$ are connected by at
  least one path in $\mathcal{P}$.  The $G$-invariant paths in
  $\mathcal{P}$ correspond to a set of paths in a finite graph (with
  loops when edges are inverted) and we say $\mathcal{P}$ is
  \emph{regular} if the set of paths over the finite alphabet of
  oriented edges in the quotient is regular in the usual sense.
  Finally, the fellow traveler property is modified to consider pairs
  of paths that start and end within $\ell$ units of each other where
  $\ell$ has been chosen large enough so that any two vertices in the
  $G$-orbit of $v$ are connected by a sequence of other vertices in
  the orbit with adjacent vertices in the sequence at most $\ell$
  units apart.  When the $G$-orbit is vertex transitive, $\ell=1$ is
  sufficient.
\end{defn}

\begin{thm}[\'{S}wi\c{a}tkowski]\label{thm:group}
  Let $G$ be a group acting geometrically on a connected graph $X$,
  and let $\mathcal{P}$ be a $G$-invariant collection of paths in $X$
  that maps onto $G$.  If $\mathcal{P}$ is regular and fellow travels
  in the modified sense described above, then $G$ is biautomatic.
\end{thm}

\section{Diagrams}\label{sec:diagrams}

The curvature of a \pe $\cat(0)$ $2$-complex impacts the distances
between paths in its $1$-skeleton via the reduced disc diagrams that
fill its closed loops.

\begin{defn}[Disc diagrams]
  A \emph{disc diagram} $D$ is a contractible $2$-complex with an
  implicit embedding into the Euclidean plane and a \emph{nonsingular
    disc diagram} is a $2$-complex homeomorphic the closed unit disc
  $\D^2$.  Disc diagrams that are not nonsingular are called
  \emph{singular} and the main difference between the two is that
  singular disc diagrams contain \emph{cut points}, i.e. points whose
  removal disconnects the diagram.  The boundary of $D$ is a
  combinatorial loop that traces the outside of $D$ in a clockwise
  fashion.  When $D$ is nonsingular, this is clear and when $D$ is
  singular amibiguities are resolved by the planar embedding.  The
  boundary cycle is the essentially unique loop in $D$ that can be
  homotoped into $\R^2 \setminus D$ while moving each point an
  arbitrarily small distance.
\end{defn}

\begin{defn}[Diagrams over complexes]
  Let $D$ be a disc diagram and let $K$ be a regular cell complex.  We
  say that a cellular map $f:D\to K$ turns $D$ into a \emph{disc
    diagram over $K$} and its boundary cycle is the image of the
  boundary cycle of $D$ under the map $f$.  When a loop $\alpha$ in
  $K$ is the boundary cycle of a disc diagram $D$ over $K$ we say that
  $\alpha$ \emph{bounds} $D$ and that $D$ \emph{fills} $\alpha$.
\end{defn}

\begin{defn}[Reduced disc diagrams]
  Let $f:D\to K$ be a disc diagram over a regular cell complex $K$.
  When there is a pair of $2$-cells in $D$ sharing an edge $e$ such
  that they are folded over along $e$ and identified under the map
  $f$, we say $D$ contains a \emph{cancellable pair} and a disc
  diagram over $K$ with no cancellable pairs is said to be
  \emph{reduced}.  Note that a disc diagram over $K$ is a reduced disc
  diagram over $K$ iff the map $f:D\to K$ is an immersion away from
  the $0$-skeleton of $D$, which is true iff for every vertex $v$ in
  $D$, the induced map from the link of $v$ to the link of $f(v)$ is
  an immersion.
\end{defn}

The importance of reduced disc diagrams over $2$-complexes derives
from the following classical result.  For a proof see \cite{McWi02}.

\begin{thm}[Van Kampen's Lemma]\label{thm:vk}
  Every combinatorial loop in a simply-connected regular cell complex
  bounds a reduced disc diagram.
\end{thm}

In $2$-dimensions, curvature conditions on the cell complex $K$
propagate to reduced disc diagrams over $K$.

\begin{cor}[$\cat(0)$ disc diagrams]\label{cor:vk}
  Every combinatorial loop in a \pe $\cat(0)$ $2$-complex bounds a \pe
  $\cat(0)$ disc diagram.
\end{cor}

\begin{proof}
  Let $\alpha$ be a combinatorial loop in a \pe $\cat(0)$ $2$-complex
  $K$.  Since $\cat(0)$ spaces are contractible, $\alpha$ bounds a
  reduced disc diagram $D$ by Theorem~\ref{thm:vk}.  By pulling back
  the metric of $K$ to $D$, we can turn it into a \pe disc diagram.
  Next, by Corollary~\ref{cor:char-npc}, $D$ is nonpositively curved.
  And finally, by Proposition~\ref{prop:char-cat} $D$ is $\cat(0)$.
\end{proof}

The fine structure of a \pe $\cat(0)$ disc diagram is highlighted when
we focus on its local curvature contributions.

\begin{defn}[Angled $2$-complexes] 
  An \emph{angled $2$-complex} is a $2$-complex $K$ where the corner
  of every polygon is assigned a number.  When $K$ is \pe, this number
  is naturally the measure of its angle, or equivalently the length of
  the spherical arc that is the link of this vertex in this polygon.
\end{defn}

\begin{defn}[Curvatures] 
  The curvature of a $n$-gon in an angled $2$-complex $K$ is the sum
  of its angles minus the expected Euclidean angle sum of $(n-2)\pi$.
  When $K$ is \pe, its polygon curvatures are all $0$.  The curvature
  of a vertex $v$ in an angled $2$-complex is $2\pi$ plus $\pi$ times
  the Euler characteristic of its link minus the sum of angles sharing
  $v$ as a vertex.  When the angled $2$-complex is a disc diagram $D$
  this reduces to $2\pi$ minus the angle sum for interior vertices and
  $\pi$ minus the angle sum for boundary vertices that are not cut
  points.  Note that when $D$ is \pe and nonpositively curved, the
  curvature of an interior vertex is nonpositive.
\end{defn}

\begin{thm}[Combinatorial Gauss-Bonnet]\label{thm:cbg}
  If $K$ is an angled $2$-complex then the sum of the vertex
  curvatures and the polygon curvatures is $2\pi$ times the Euler
  characteristic of $K$.
\end{thm}

A proof of this relatively elementary result can be found in
\cite{McWi02}.  The basic idea is that every assigned angle
contributes to exactly one vertex curvature and exactly one polygon
curvature but with opposite signs.  The curvature sum is thus
independent of the angles assigned and minor bookkeeping equates it
with $2\pi$ times the Euler characteristic.  The main result we need
is a easy corollary.

\begin{cor}\label{cor:disc}
  If $D$ is a \pe $\cat(0)$ disc diagram, then the sum of its boundary
  vertex curvatures is at least $2\pi$.
\end{cor}

\begin{proof}
  This follows from Theorem~\ref{thm:cbg} once we note that being a
  disc diagram implies its Euler characteristic is $1$, \pe implies
  its face curvatures are $0$ and $\cat(0)$ implies its interior
  vertex curvatures are nonnegative.
\end{proof}

\section{Geodesics}\label{sec:geos}

In this section we use disc diagrams to turn an arbitrary path into a
geodesic inside a $\cat(0)$ triangle-square complex.  This procedure
uses moves that modify paths by pushing across simple diagrams.  After
introducing these diagrams and moves, we describe the procedure.

\begin{defn}[Doubly-based disc diagrams]\label{def:doubly-based}
  A \emph{doubly-based disc diagram} is a nonsingular disc diagram $D$
  with two distinquished boundary vertices $u$ and $v$.  An example is
  shown in Figure~\ref{fig:sample-disc}.  The vertex $u$ is its
  \emph{start vertex} and $v$ is its \emph{end vertex}.  When $u$ and
  $v$ are distinct there are two paths in the boundary cycle of $D$
  that start at $u$ and end at $v$.  The one that proceeds clockwise
  along the boundary is the \emph{old path} and the one that proceeds
  counter-clockwise along the boundary is \emph{new path}.  When
  illustrating doubly-based diagrams, we place $u$ on the left and $v$
  on the right so that the old path travels along the top and the new
  path travels along the bottom.  The old path/new path terminology is
  extended to the case $u=v$ by defining the old path as the clockwise
  boundary cycle starting and ending at $u=v$ and the new path as the
  trivial path at $u=v$.
\end{defn}

\begin{figure}
  \includegraphics{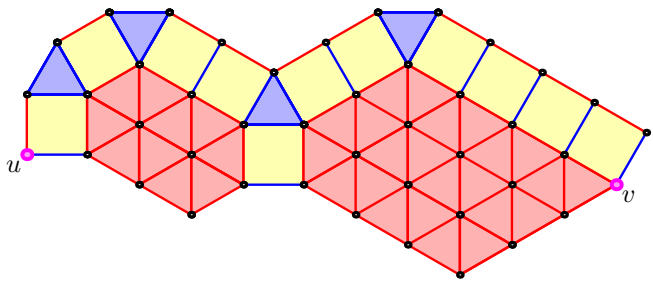}
  \caption{A doubly-based $\cat(0)$ disc
    diagram.\label{fig:sample-disc}}
\end{figure}

\begin{defn}[Moves]\label{def:moves}
  Let $\gamma$ be a path in a $2$-complex $K$ and let $D$ be a
  doubly-based disc diagram.  If there is a map $D\to K$ that sends
  the old path to a directed subpath of $\gamma$, then \emph{applying
    the move $D$} means replacing the image of the old path in
  $\gamma$ with the image of the new path to create an altered path
  $\gamma'$.  This can be thought of as homotoping $\gamma$ across the
  image of $D$.  The move is \emph{length-preserving} when the old and
  new paths have the same length and \emph{length-reducing} when the
  new path is strictly shorter.
\end{defn}

The moves we are interested in are quite specific.

\begin{defn}[Basic moves]\label{def:basic-moves}
  We define five basic types of moves over a triangle-square complex
  that involve a single triangle, a single square, a pair of
  triangles, a finite row of squares with a triangle on either end, or
  a single edge.  We call these \emph{triangle moves}, \emph{square
    moves}, \emph{triangle-triangle moves},
  \emph{triangle-square-triangle moves}, and \emph{trivial moves},
  respectively.  Diagrams illustrating the first four are shown in
  Figure~\ref{fig:moves}.  Note that for each type of move the old
  path is at least as long as the new path.  The triangle move is
  length-reducing.  The start and end vertices for the square move are
  non-adjacent making it length-preserving.  The start and end
  vertices of a triangle-triangle move or a triangle-square-triangle
  move are the unique vertices of valence~$2$ and they are also
  length-preserving.  Note that a triangle-square-triangle move can
  have any number of squares and that a triangle-triangle move can be
  viewed as \emph{degenerate triangle-square-triangle} move with zero
  squares.  Finally, the trivial move (defined below) is not a move in
  the sense of Definition~\ref{def:moves}, but it is included for
  convenience.  Let $D$ be a diagram consisting of a single edge and
  identify both $u$ and $v$ with one of its vertices.  If we view the
  path of length~$2$ starting and ending at $u=v$ as the old path and
  the trivial path at $u=v$ as the new path, then we can view the
  removal of a ``backtrack'' in a path $\gamma$ as applying a trivial
  move across $D$.
\end{defn}

Because the only moves considered from this point forward are the
basic moves described above, we drop the adjective.  Thus a move now
refers to one of these five basic types.  Notice that for every move
(in this revised sense), the subspace distance between the old and new
paths is $1$ since every vertex lies in one of the paths and, if not
in both, it is connected to a vertex in the other path by an edge.

\begin{figure}
  \begin{tabular}{cccc}
    \begin{tabular}{c}\includegraphics{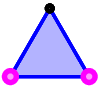}\end{tabular}&
    \begin{tabular}{c}\includegraphics{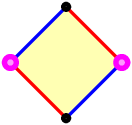}\end{tabular}& 
    \begin{tabular}{c}\includegraphics{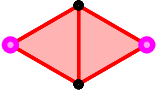}\end{tabular}& 
    \begin{tabular}{c}\includegraphics{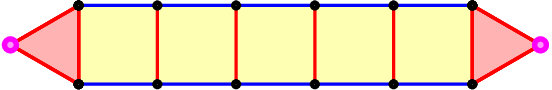}\end{tabular}\\  
  \end{tabular}\\
  \caption{A triangle move, square move, triangle-triangle move and an
    example of a triangle-square-triangle move.\label{fig:moves}}
\end{figure}

\begin{defn}[Edges and paths]
  Let $D$ be a nonsingular triangle-square disc diagram.  Every
  boundary edge $e$ belongs to a unique $2$-cell in $D$ and we call
  $e$ a \emph{triangle edge} or \emph{square edge} depending on the
  type of this cell.  Unless the boundary edges are all of one type,
  the boundary cycle of $D$ can be divided into an alternating
  sequence of maximal triangle edge paths and maximal square edge
  paths that we call \emph{triangle paths} and \emph{square paths}.
  Finally, a boundary vertex with positive curvature that is not the
  transition between a square path and a triangle path is said to be
  \emph{exposed}.
\end{defn}

The disc diagram shown in Figure~\ref{fig:sample-disc} has $12$
triangle edges and $13$ square edges, its boundary is divided into $5$
triangle paths and $5$ square paths, and there are $4$ exposed
vertices, one of which is $u$.  Note that $v$ is not exposed even
though its curvature is positive.

\begin{lem}[Exposed vertices]\label{lem:exposed-vertices}
  If $D$ is a nonsingular triangle-square disc diagram with exposed
  vertex $v$, then a triangle move, square move, or triangle-triangle
  move can be applied to the boundary path of length~$2$ with $v$ as
  its middle vertex.
\end{lem}

\begin{proof}
  Because every angle in a triangle-square complex is $\frac{\pi}{3}$
  or $\frac{\pi}{2}$, the possibilities for positive curvature are
  very limited.  By nonsingularity there is at least one angle at $v$
  but, if positively curved, there are at most two angles at $v$.  The
  only possibilities are one triangle, one square, or two triangles.
  Two squares lead to nonpositive curvature and one of each only
  happens at the transitions between triangle and square paths.  In
  each feasible case there is an obvious move.
\end{proof}

\begin{defn}[Cumulative curvatures]
  Let $D$ be a nonsingular triangle-square disc diagram $D$ with both
  triangle and square edges.  The sum the curvatures of the boundary
  vertices along a square path in the boundary of $D$, including its
  endpoints, is called the \emph{cumulative curvature} of this path.
\end{defn}

If we let $\alpha$ and $\beta$ denote the old and new paths
respectively of the doubly-based disc diagram shown in
Figure~\ref{fig:sample-disc}, then the cumulative curvatures of its
$5$ square paths (starting at $u$ and continuing clockwise) are
$\frac{\pi}{2}$, $\frac{\pi}{3}$, $0$, $\frac{5\pi}{6}$ and
$-\frac{\pi}{3}$.  These add the curvatures from $\beta(1)$, $u$ and
$\alpha(1)$, from $\alpha(2)$ and $\alpha(3)$, from $\alpha(4)$
through $\alpha(8)$, from $\alpha(9)$ through $v = \beta(11)$ and
finally from $\beta(5)$ and $\beta(4)$, respectively.

\begin{lem}[Square paths]\label{lem:square-paths}
  If $D$ is a nonsingular triangle-square disc diagram and $\alpha$ is
  a square path in its boundary containing $k$ exposed vertices, then
  the cumulative curvature of $\alpha$ is at most
  $\frac{\pi}{3}+k\frac{\pi}{2}$.  Moreover, if $\alpha$ contains no
  exposed vertices and has positive cumulative curvature, then a
  triangle-square-triangle move can be applied to a slight extension
  of $\alpha$.
\end{lem}

\begin{proof}
  The endpoints of $\alpha$ have curvature at most $\frac{\pi}{6}$ (as
  they are incident with both a triangle and a square) and the
  interior vertices of $\alpha$ are either exposed with curvature
  $\frac{\pi}{2}$ or not exposed and nonpositively curved.  The first
  assertion simply tallies these bounds.  For the second, note that
  with no exposed vertices the cumulative curvature is at most
  $\frac{\pi}{3}$.  If any unexpected angle (i.e. one that is not from
  a triangle or square that contains an edge of the path) is incident
  with a vertex of $\alpha$ then its cumulative curvature drops at
  least to $0$.  And if no unexpected angles exist, $\alpha$ plus the
  triangle edges before and after form the old path of the
  triangle-square-triangle move whose $2$-cells are the cells of $D$
  containing these boundary edges.
\end{proof}

The key result we need is the following.

\begin{prop}[Moves and basepoints]\label{prop:exposed-moves}
  If $D$ is a doubly-based nonsingular $\cat(0)$ triangle-square disc
  diagram, then a move can be applied to either the old path or the
  new path.
\end{prop}

\begin{proof}
  If there is an exposed vertex distinct from $u$ and $v$, then
  Lemma~\ref{lem:exposed-vertices} produces the required move.  Note
  that this case already covers all diagrams without both triangle
  edges and square edges.  With no transitions between triangle paths
  and square paths, every positively curved boundary vertex is
  exposed, its curvature is at most $\frac{2\pi}{3}$ and, since the
  total is at least $2\pi$ (Corollary~\ref{cor:disc}), $D$ contains at
  least $3$ exposed vertices.  

  Next, if no exposed vertex distinct from $u$ and $v$ exists but
  there is a square path with positive cumulative curvature disjoint
  from $u$ and $v$, then Lemma~\ref{lem:square-paths} produces the
  required move.  This second case covers all remaining diagrams, a
  fact we prove by contradiction.

  Suppose $D$ is a diagram satisfying neither condition.  Because $D$
  contains both triangle edges and square edges, its boundary can be
  divided into triangle paths and square paths.  Since there are no
  exposed vertices distinct from $u$ and $v$, every other vertex in
  the interior of a triangle path is nonpositively curved.  By
  Corollary~\ref{cor:disc} the cumulative curvatures of the square
  paths plus the curvatures of the vertices in the interior of the
  triangle paths must be at least $2\pi$, but the only positive
  summands come from $u$, $v$ and/or the square paths that contains
  $u$, $v$ or both.  Each possible configuration falls short.  For
  example, if neither $u$ nor $v$ belong a square path, $u$ and $v$
  each contribute at most $\frac{2\pi}{3}$ and there are no other
  positive summands.  Similarly, if one basepoint belongs to a square
  path, but the other does not, then the square path containing the
  basepoint has cumulative curvature at most
  $\frac{\pi}{3}+\frac{\pi}{2}$ (Lemma~\ref{lem:square-paths}), the
  other basepoint contributes at most $\frac{2\pi}{3}$ and there are
  no other positive summands.  Next, if $u$ and $v$ belong the same
  square path, its cumulative curvature is at most
  $\frac{\pi}{3}+2\frac{\pi}{2}$ and there are no other positive
  summands.  Finally, if $u$ and $v$ belong the distinct square paths,
  their cumulative curvatures are each at most
  $\frac{\pi}{3}+\frac{\pi}{2}$ and there are no other positive
  summands.  In each case the total is less than $2\pi$,
  contradiction.
\end{proof}

\begin{thm}[Straightening paths]\label{thm:straightening} 
  Let $K$ be a $\cat(0)$ triangle-square complex and let $\alpha$ and
  $\beta$ be combinatorial paths in $K$ that start at $u$ and end at
  $v$, two not necessarily distinct vertices in $K$.  When $\beta$ is
  a (possibly trivial) geodesic, $\alpha$ can be transformed into
  $\beta$ by a finite sequence of length-preserving and
  length-reducing moves.  When both paths are geodesics, only
  length-preserving moves are needed.
\end{thm}

\begin{proof}
  The proof proceeds by induction primarily on length and secondarily
  on area.  More specifically, we show that for every coterminous pair
  of paths $\alpha$ and $\beta$, we can either apply a length-reducing
  move to $\alpha$, or a length-preserving move to $\alpha$ or $\beta$
  so that the new pair of paths bounds a disc diagram with strictly
  fewer $2$-cells than the smallest diagram that $\alpha\beta^{-1}$
  bounds.  First the easy cases. If $\alpha=\beta$ there is nothing to
  prove.  If $\alpha$ is not immersed then it can be shortened by a
  trivial move.  If $\alpha$ is immersed but not embedded then a
  closed proper subpath $\alpha'$ can be paired with a trivial path
  satisfying the induction hypothesis.  (The path $\beta$ is always
  immersed and embedded because it is geodesic.)  Similarly, if
  $\alpha$ and $\beta$ share internal vertices, then we can split
  $\alpha$ into $\alpha_1 \alpha_2$ and $\beta$ into $\beta_1 \beta_2$
  where $\alpha_i$ and $\beta_i$ are shorter paths satisfying the
  induction hypothesis.

  In the remaining case, $\alpha \beta^{-1}$ bounds a $\cat(0)$ disc
  diagram $D$ (Corollary~\ref{cor:vk}) and the restrictions on
  $\alpha$ and $\beta$ force $D$ to be nonsingular.  By
  Proposition~\ref{prop:exposed-moves} there is either a
  length-reducing or length-preserving diagram-shrinking move that
  avoids the distinguished vertices $u$ and $v$.  This is sufficient
  since length-preserving moves can only occur a bounded number of
  times before one of the easy cases or a length-reducing move occurs.
  Finally, since $\beta$ is a geodesic, the modifications to $\beta$
  and its descendents only involve length-preserving moves.  Thus, the
  modifications can be reordered so that all the alterations to paths
  descended from $\alpha$ happen first, followed by the alterations to
  the paths descended from $\beta$ in reverse (and in reverse order)
  to complete the process.
\end{proof}

\section{Intervals}\label{sec:intervals}

We now shift our attention from a single geodesic to the global
structure of all geodesics connecting fixed vertices in a $\cat(0)$
triangle-square complex.

\begin{figure}
  \includegraphics{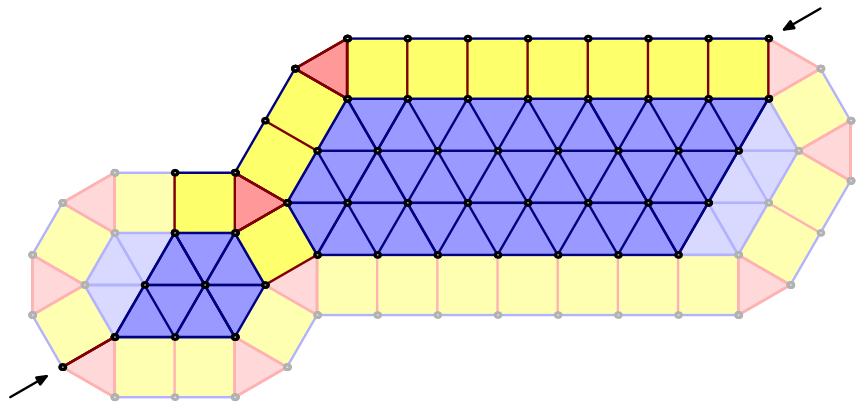}
  \caption{A doubly-based $\cat(0)$ disc diagram with the
    corresponding interval heavily shaded.\label{fig:interval}}
\end{figure}

\begin{defn}[Intervals]
  Let $u$ and $v$ be fixed vertices in a $\cat(0)$ triangle-square
  complex $K$ that are distance $n$ apart in the $1$-skeleton metric.
  We say a vertex $w$ is \emph{between $u$ and $v$} when $w$ lies on a
  combinatorial geodesic connecting $u$ to $v$ or, equivalently, when
  $d(u,w) + d(w,v) = d(u,v)$.  The \emph{interval} between $u$ and
  $v$, denoted $K[u,v]$, is the largest subcomplex of $K$ with all
  vertices between $u$ and $v$, i.e. the \emph{full subcomplex} on
  this vertex set.  More explicitly, it includes all vertices between
  $u$ and $v$, all of the edges of $K$ connecting them, and every
  triangle or square in $K$ with all of its vertices in this set.
  Note that $K[u,v]$ contains every geodesic from $u$ to $v$ as well
  as every length-preserving move that converts one such geodesic into
  another one.  An example of an interval is shown in
  Figure~\ref{fig:interval}.
\end{defn}

\begin{rem}[Planar and locally Euclidean]
  Examples that are planar and locally Euclidean can be illustrated
  with undistorted metrics but neither of these properties hold in
  general for intervals in $\cat(0)$ triangle-square complexes.  See
  Figure~\ref{fig:bad-intervals}.  This is one place where the
  ``mixed'' case is very different from the ``pure'' cases.  Gersten
  and Short proved, at least implicitly, that intervals in $\cat(0)$
  square complexes can always be identified with subcomplexes of the
  standard square tiling of $\R^2$ and that intervals in $\cat(0)$
  triangle complexes can always be identified with subcomplexes of the
  standard triangle tiling of $\R^2$.  In particular they are always
  planar and locally Euclidean.  In fact, B.~T.~Williams proved in his
  dissertation that an analogous result holds for $n$-dimensional cube
  complexes and the standard cubing of $\R^n$ \cite{Wil98}.
\end{rem}

\begin{figure}[h]
  \begin{tabular}{ccc}
  \begin{tabular}{c}\includegraphics{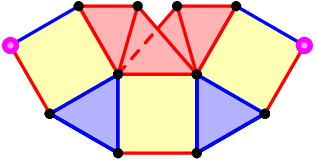}\end{tabular} & 
  \hspace*{2em} &
  \begin{tabular}{c}\includegraphics{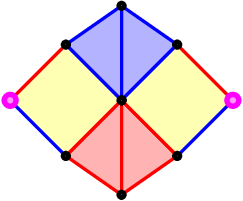}\end{tabular} 
  \end{tabular}
  \caption{An interval that is not planar and another that is not locally
    Euclidean.\label{fig:bad-intervals}}
\end{figure}

Despite these complications, an interval in a $\cat(0)$
triangle-square complex is still a highly structured object.  As we
now show, it is a simply-connected $\cat(0)$ subcomplex in which every
edge of its $1$-skeleton belongs to a either a geodesic or a move.  We
begin with a definition.

\begin{defn}[Vertical and horizontal edges]\label{def:v/h-edges}
  Let $K[u,v]$ be an interval in a $\cat(0)$ triangle-square complex
  $K$ and let $x$ and $y$ be vertices in $K[u,v]$ connected by an edge
  $e$.  We call $e$ a \emph{horizontal edge} of $K[u,v]$ when $x$ and
  $y$ are the same distance from $u$ (and consequently the same
  distance from $v$) and a \emph{vertical edge} otherwise.
\end{defn}

\begin{lem}[Edges in intervals]\label{lem:interval-edge}
  Let $K[u,v]$ be interval in a $\cat(0)$ triangle-square complex $K$.
  Every edge in the interval either lies on a geodesic from $u$ to $v$
  or is part of a move between two such geodesics.  More specifically,
  vertical edges lies on geodesics and horizontal edges are contained
  in moves.
\end{lem}

\begin{proof}
  Let $e$ be an edge connecting $x$ and $y$ in $K[u,v]$.  Because $x$
  and $y$ are in $K[u,v]$ they lie on geodesics $\alpha$ and $\beta$
  and can be viewed as $x=\alpha(i)$ and $y=\beta(j)$ for integers $i$
  and $j$.  These integers also represent the distance $x$ and $y$ are
  from $u$, respectively.  The difference between $i$ and $j$ is at
  most $1$ because otherwise, we could use one end of $\alpha$, $e$
  and the other end of $\beta$ to construct a new path shorter than
  the geodesics $\alpha$ and $\beta$.  As it is, if $i+1=j$, the path
  that follows $\alpha$ from $u$ to $x=\alpha(i)$, crossses $e$, and
  then follows $\beta$ from $y=\beta(j)$ to $v$ is a geodesic
  containing $e$.
  
  When $e$ is horizontal, add a new triangle to $K$ along $e$ with
  third vertex $z$.  It is easy to check that the augmented complex
  remains $\cat(0)$ and that there are geodesics $\alpha_1$ and
  $\beta_1$ from $u$ to $z$ through $x$ and $y$, respectively.  By
  Theorem~\ref{thm:straightening} there is a finite sequence of
  length-preserving moves that transforms $\alpha_1$ to $\beta_1$
  inside the augmented complex.  Since the only $2$-cell that contains
  $z$ is a triangle any move that converts a geodesic through $x$ to a
  geodesic through $y$ must be a (possibly degenerate)
  triangle-square-triangle move.  In particular, $K[u,v]$ contains a
  triangle-square-triangle move minus its final triangle pointing
  towards $u$ and ending at $e$.  After repeating this argument
  focusing on the beginning of the interval $K[z,v]$, we can merge the
  two uncapped triangle-square-triangle moves containing $e$, one
  pointing towards $u$ and the other towards $v$, into a
  triangle-square-triangle move inside $K[u,v]$ that contains $e$.
\end{proof}

\begin{thm}[Intervals are $\cat(0)$]\label{thm:intervals-cat0}
  Intervals in $\cat(0)$ triangle-square complexes are
  simply-connected and thus $\cat(0)$.
\end{thm}

\begin{proof}
  Let $K[u,v]$ be an interval in a $\cat(0)$ triangle-square complex
  $K$.  As a subcomplex of a $\cat(0)$ \pe $2$-complex, $K[u,v]$ is
  nonpositively curved by Corollary~\ref{cor:char-npc}.  Once we
  establish that it is simply-connected, $\cat(0)$ follows by
  Proposition~\ref{prop:char-cat}.
  Let $\gamma$ be a closed combinatorial path in $K[u,v]$.  Using
  Lemma~\ref{lem:interval-edge} we can embed each edge in $\gamma$ in
  either a geodesic from $u$ to $v$ or a move connecting two such
  geodesics.  Theorem~\ref{thm:straightening} can then be used to fill
  in the gaps between these geodesics and the geodesics on either side
  of these moves in such a way that there is a map $\sph^2 \to K[u,v]$
  sending the south pole to $u$, the north pole to $v$, the equator to
  $\gamma$ (with constant maps inserted between the edges), longitude
  lines to geodesics and lunes between the adjacent longitude lines
  alternately to the diagrams provided by
  Theorem~\ref{thm:straightening} and the geodesic or move containing
  a particular edge of $\gamma$.  The image of the southern hemisphere
  under this map shows that $\gamma$ is null-homotopic and $K[u,v]$ is
  simply-connected.
\end{proof}

\section{Gersten-Short geodesics}\label{sec:gs-geos}

In this section we define a collection of paths in a $\cat(0)$
triangle-square complex that reduces to the paths chosen by Gersten
and Short in $\cat(0)$ triangle complexes and in $\cat(0)$ square
complexes.  The first step is to show that all geodesics from $u$ to
$v$ either contain a common first edge or there is a unique nontrivial
move in the interval starting at $u$.

\begin{prop}[Unique $2$-cell]\label{prop:unique-cell}
  If both the old path and the new path of a doubly-based nonsingular
  $\cat(0)$ triangle-square disc diagram $D$ are geodesics, then the
  start vertex of $D$ lies in the boundary of a unique $2$-cell.
\end{prop}

\begin{proof}
  Suppose not and select a counterexample that is minimal in the
  following sense.  First minimize the distance between the basepoints
  and then minimize the number of $2$-cells in $D$.  Any move that can
  be applied to either path in such a minimal counterexample must
  contain the vertex $u$ as an endpoint since applying a move not
  containing $u$ either produces a counterexample with fewer $2$-cells
  or a singular diagram containing a counterexample whose basepoints
  are closer together.

  If all boundary edges of $D$ are square edges, then every positively
  curved boundary vertex is exposed with curvature $\frac{\pi}{2}$.
  To reach $2\pi$ (Corollary~\ref{cor:disc}) there must be at least
  $4$ such vertices. Since $u$ itself is not exposed, there exists an
  exposed vertex distinct from $u$, $v$ and the vertices adjacent to
  $u$, but this leads to a square move contradicting the minimality of
  $D$ (Lemma~\ref{lem:exposed-vertices}).  Thus triangle edges exist
  in $D$.

  Next, add the boundary vertex curvatures of $D$ so that the
  curvature of each boundary vertex not touching a square edge is an
  individual summand and the curvature of vertices in square paths are
  collected together into cumulative curvature summands.  The total is
  at least $2\pi$ (Corollary~\ref{cor:disc}), but the only positive
  summands come from $u$ and $v$ (and/or the square paths that contain
  them) and from moves that contain $u$ as an endpoint since other
  positive terms lead to moves that contradict the minimality of $D$
  (Lemmas~\ref{lem:exposed-vertices} and~\ref{lem:square-paths}).

  When both boundary edges touching $u$ are triangle edges, the
  curvature of $u$ is at most $\frac{\pi}{3}$ and any move with $u$ as
  an endpoint is a (possibly degenerate) triangle-square-triangle move
  with cumulative curvature at most $\frac{\pi}{3}$.  Note that there
  are at most two such moves, one in each path.  Finally, the
  curvature of $v$ or the cumulative curvature of the square path
  containing $v$ is at most $\frac{5\pi}{6}$, the extreme case being
  $v$ exposed in the middle of a square path.  These are the only
  possible sources of positive curvature and their total is bounded by
  $\frac{11\pi}{6} < 2\pi$, contradiction.  

  Next, suppose $u$ and $v$ belong the same square path.  If $u$ is
  not an endpoint, then the only positive summand is the cumulative
  curvature of this square path.  Moreover, its curvature is at most
  $\frac{\pi}{3} + 3 \cdot \frac{\pi}{2}$
  (Lemma~\ref{lem:square-paths}) since the $v$ and the two vertices
  adjacent to $u$ are the only ones that can be exposed.  If $u$ and
  $v$ belong to the same square path and $u$ is an endpoint, then the
  only positive summands are the cumulative curvature of this square
  path and a possible triangle-square-triangle move ending at $u$.
  These curvatures are bounded by $\frac{\pi}{3} + 2 \cdot
  \frac{\pi}{2}$ and $\frac{\pi}{3}$, respectively.  Both totals are
  less than $2\pi$, contradiction.
  
  Finally, suppose $u$ belongs to a square path that does not include
  $v$. As above, the curvature of $v$ or the cumulative curvature of
  the square path containing $v$ is at most $\frac{5\pi}{6}$, with the
  extreme case being $v$ exposed in the middle of a square path.  Thus
  it suffices to show that the curvatures associated with $u$ are at
  most $\pi$.  If $u$ is not an endpoint of the square path, then $u$
  itself is nonpositively curved and total curvature of the path on
  either side is bounded by $\frac{\pi}{3}$. To see this note that an
  easy bound is $\frac{\pi}{6} + \frac{\pi}{2}$, the first term from a
  positively curved endpoint and the second from a single exposed
  vertex adjacent to $u$, but if there are no unexpected angles along
  the path, then this leads a configuration that violates our
  assumption that both boundary paths from $u$ to $v$ are geodesics.
  Thus the improved bound is $\frac{\pi}{6} + \frac{\pi}{2}
  -\frac{\pi}{3} = \frac{\pi}{3}$.  And if $u$ is endpoint, then the
  cumulative curvature of the square path excluding $u$ is bounded by
  $\frac{\pi}{3}$ using the argument just given.  The curvature of $u$
  is at most $\frac{\pi}{6}$ and the triangle-square-triangle move
  ending at $u$, if it exists, has curvature at $\frac{\pi}{3}$.
  Thus, all configurations lead to totals less than $2\pi$.
\end{proof}

\begin{prop}[Initial link]\label{prop:initial-link}
  If $u$ and $v$ are distinct vertices in a $\cat(0)$ triangle-square
  complex $K$, then the link of $u$ in the interval $K[u,v]$ is either
  a single vertex or a single edge.
\end{prop}

\begin{proof} 
  Every pair of distinct vertices in the link corresponds to edges
  leaving $u$ that can be extended to geodesics $\alpha$ and $\beta$
  from $u$ to $v$. From these geodesics
  Theorem~\ref{thm:straightening} implicitly constructs a path in the
  link connecting the original vertices.  In fact, by
  Corollary~\ref{cor:disc} there is a $\cat(0)$ disc diagram $D$
  filling $\alpha\beta^{-1}$.  After restricting our attention to a
  nonsingular subdiagram containing $u$ if necessary,
  Proposition~\ref{prop:unique-cell} shows that there is a $2$-cell in
  $D$ touching $u$ providing an edge in the link connecting the two
  original vertices.  If the link contains three distinct vertices,
  then edges connect them to form a triangular geodesic loop that is
  short since each edge has length at most $\frac{\pi}{2}$.  This is
  impossible because $K[u,v]$ is $\cat(0)$
  (Theorem~\ref{thm:intervals-cat0}).  Finally, loops of length~$1$
  and multiple edges are prohibited in the link for similar reasons.
\end{proof}

\begin{thm}[First move]\label{thm:first-move}
  If $u$ and $v$ are distinct vertices in a $\cat(0)$ triangle-square
  complex $K$, then either every geodesic from $u$ to $v$ has the same
  first edge or there are exactly two possible first edges and there
  is a unique move in $K[u,v]$ that contains both of them.
\end{thm}

\begin{proof}
  When the link of $u$ in $K[u,v]$ is a single vertex, then the
  corresponding edge in $K[u,v]$ is the first edge in every geodesic
  from $u$ to $v$.  The only other possibility, according to
  Proposition~\ref{prop:initial-link} is that the link of $u$ in
  $K[u,v]$ is a single edge.  In this case the vertex $u$ lies in the
  boundary of a unique $2$-cell and both edges adjacent to $u$ can be
  extended to geodesics from $u$ to $v$.  That this $2$-cell is part
  of some move in $K[u,v]$ is a consequence of
  Theorem~\ref{thm:straightening}.  To see uniqueness, suppose there
  were two such moves and focus on the $2$-cells closest to $u$ that
  are in one move but not the other.  Using
  Proposition~\ref{prop:initial-link} to fill in gaps between vertical
  edges with a common start vertex it is easy to construct a short
  loop in its link, contradicting the fact that $K[u,v]$ is $\cat(0)$.
\end{proof}

\begin{defn}[Gersten-Short geodesics]\label{def:gs-geos}
  Let $u$ and $v$ be vertices in a $\cat(0)$ triangle-square complex
  $K$.  A \emph{Gersten-Short geodesic} from $u$ to $v$ is defined
  inductively using Theorem~\ref{thm:first-move}.  If $u=v$ then the
  only geodesic is the trivial path.  When $u$ and $v$ are distinct
  there are two possibilities.  If every geodesic from $u$ to $v$ has
  the same first edge then the Gersten-Short geodesic travels along
  this edge to its other endpoint $u'$, and if there is a unique move
  in $K[u,v]$ that contains the two possible first edges, then we
  travel along either side of this unique move to its other endpoint
  $u'$.  In either case, the new vertex $u'$ is closer to $v$ and the
  rest of the path is an inductively defined Gersten-Short geodesic
  from $u'$ to $v$.  The vertex $u'$ is called the \emph{first choke
    point} of $K[u,v]$.  If we let $u_0=u$ and define $u_{i+1}$ as the
  first choke point of $K[u_i,v]$ then eventually $u_i=v$.  These
  $u_i$'s are the \emph{choke points} between $u$ and $v$ and every
  Gersten-Short geodesic from $u$ to $v$ goes through each one.
\end{defn}

\begin{figure}
  \includegraphics{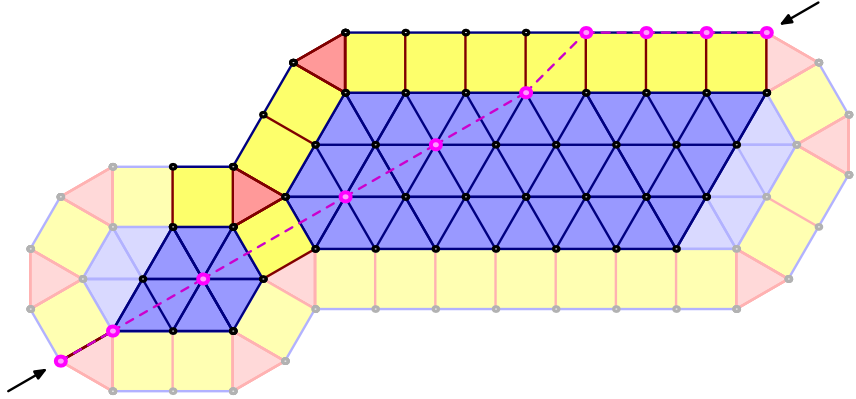}
  \caption{The choke points of the Gersten-Short geodesics in an
    interval.\label{fig:gs-geos}}
\end{figure}

\begin{exmp}[Gersten-Short geodesics]
  Figure~\ref{fig:gs-geos} shows the series of choke points in an
  interval used to define its Gesten-Short geodesics.  When the dotted
  line connecting two choke points runs along an edge, this edge
  belongs to all of the Gersten-Short geodesics between $u$ and $v$.
  When it cuts through a number of $2$-cells, these $2$-cells form the
  unique next move and every Gersten-Short geodesic runs along one
  side or the other.  In the example, there are, in order, an edge, a
  triangle-triangle move, a triangle-square-triangle move, two more
  triangle-triangle moves, a square move and finally three edges.
  Because there are $5$ moves involved there are $2^5=32$
  Gersten-Short geodesics.
\end{exmp}

\begin{exmp}[Pure complexes]
  When $K$ is a $\cat(0)$ triangle complex or a $\cat(0)$ square
  complex, its intervals look like subcomplexes of the standard square
  and triangular tilings of $\R^2$, respectively.  It is thus
  instructive to consider what intervals and Gersten-Short geodesics
  look like when $K$ itself is one of these tilings.  Nonsingular
  intervals in this special case look like parallelograms and
  rectangles and the choke points of typical Gersten-Short paths are
  shown in Figure~\ref{fig:gs-pure}.  In simple examples such as these
  it is easy to see the directed nature of the definition.  The
  Gersten-Short geodesics from $v$ to $u$ are definitely different.
\end{exmp}

\begin{figure}
  \begin{center}
    \begin{tabular}{cc}
      \includegraphics[scale=0.7]{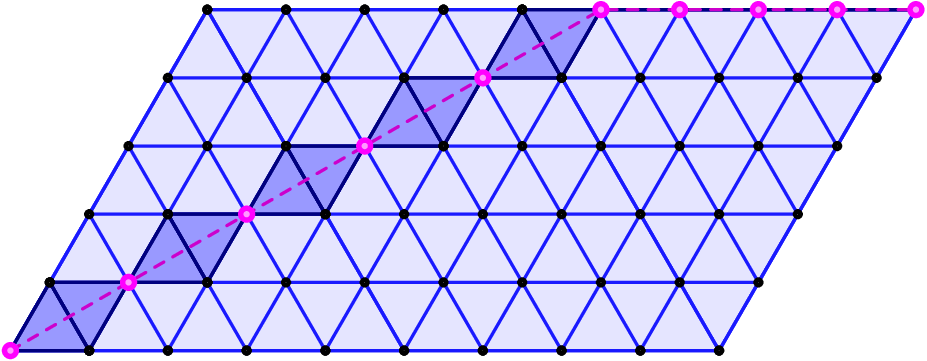} & 
      \includegraphics[scale=0.7]{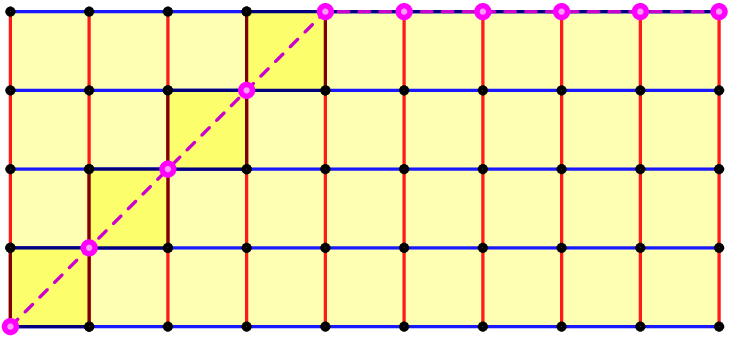}
    \end{tabular}
  \end{center}
  \caption{The choke points of Gersten-Short geodesics in intervals of
    the standard triangle and square tilings of
    $\R^2$.\label{fig:gs-pure}}
\end{figure}

It is important to remark, that in these pure cases, the paths that we
have identified are essentially the same paths that Gersten and Short
identified in \cite{GeSh90}.  The only difference is that they added
extra restrictions so that a \emph{unique} geodesic is selected
between every pair of endpoints $u$ and $v$.  This difference is
insignificant because of the following result which is an immediate
corollary of the definition and the fact that the two sides of a
length-preserving move $1$-fellow travel.

\begin{cor}[$1$-fellow travel]\label{cor:paths-fellow-travel}
  The Gersten-Short geodesics from $u$ to $v$ in a $\cat(0)$
  triangle-square complex $K$ pairwise $1$-fellow travel.
\end{cor}

\section{Flats}\label{sec:flats}

Now that we have established the existence of a language of geodesic
paths that generalizes the regular language of geodesics used by
Gersten and Short, we shift our attention to flats, the main obstacles
to establishing biautomaticity.

\begin{defn}[Flats]\label{def:flats}
  A \emph{triangle-square flat} is a triangle-square complex isometric
  to the Euclidean plane $\euc = \R^2$.  Suggestive portions of
  several triangle-square flats are shown in Figure~\ref{fig:flats}.
  More generally, a \emph{flat} is any metric space isometric to a
  Euclidean space $\R^n$, with $n$ at least $2$, and a \emph{flat in a
    metric space $K$} is a flat $F$ together with an isometric
  embedding $F\to K$.  When $F$ is a triangle-square flat and $K$ is a
  triangle-square complex we further insist that the embedding is a
  cellular map.
\end{defn}

The importance of flats is highlighted by the following result.

\begin{thm}[Flats and biautomaticity]\label{thm:flats-biaut}
  Let $K$ be a compact nonpositively curved \pe complex.  If the
  universal of $K$ does not contain an isometrically embedded flat
  plane, then $\widetilde K$ is $\delta$-hyperbolic, $\pi_1(K)$ is
  word hyperbolic, and thus $\pi_1(K)$ is biautomatic.
\end{thm}

\begin{proof}
  These assertions follow from the flat plane theorem \cite{BrHa99},
  the definition of word hyperbolicity, and the fact, proved in
  \cite{ECHLPT92}, that every word hyperbolic group is biautomatic.
\end{proof}

More generally, Chris Hruska has shown in \cite{Hr05} that any group
that acts geometrically on a $\cat(0)$ $2$-complex with the isolated
flats property is biautomatic.  For $\cat(0)$ $2$-complexes, the
isolated flats property is the same as the non-existence of an
isometric embedding of a \emph{triplane}, a triple of upper
half-planes sharing a common boundary line.  Thus, if $\widetilde K$
contains no flats (or if those that do exist are isolated), then
$\pi_1(K)$ is already known to be biautomatic.  In particular, when
biautomaticity fails, it fails directly or indirectly because of the
flats in $\widetilde K$.  When analyzing a flat, we begin by
decomposing it into regions.

\begin{defn}[Regions]\label{def:regions}
  Let $F$ be a triangle-square flat and consider the equivalence
  relation on its $2$-cells generated by placing $2$-cells in the same
  equivalence class when they are the same type and share an edge.
  The union of the closed $2$-cells in an equivalence class is called
  a \emph{region} of $F$ and it is \emph{triangular} or \emph{square}
  depending on the type of $2$-cell it contains.  For example, the
  flat in the lower lefthand corner of Figure~\ref{fig:flats} has six
  visible triangular regions and five square regions.
\end{defn}

The next step is to observe that these regions are convex, based on
the strong local restrictions that vertices in triangle-square flats
must satisfy.

\begin{figure}
  \begin{center}
    \begin{tabular}{cccc}
      \begin{tabular}{c}\includegraphics{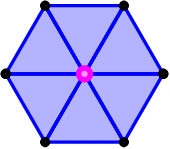}\end{tabular}&
      \begin{tabular}{c}\includegraphics{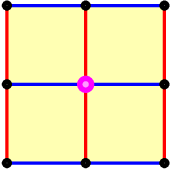}\end{tabular}&
      \begin{tabular}{c}\includegraphics{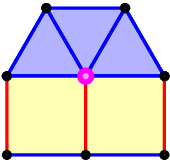}\end{tabular}&
      \begin{tabular}{c}\includegraphics{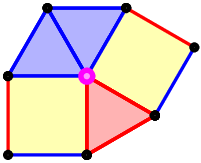}\end{tabular}\\
    \end{tabular}
  \end{center}
  \caption{Four possible vertex neighborhoods in a
    flat.\label{fig:vtypes}}
\end{figure}

\begin{defn}[Vertices in flats]
  A vertex $v$ in a triangle-square flat $F$ is surrounded by corners
  whose angles total $2\pi$ and the possibilities are extremely
  limited.  If $v$ lies in the interior of a region then it is either
  surrounded by $6$~triangles or by $4$~squares.  The only other
  possibility is that there are $3$~triangles and $2$~squares touching
  $v$ and they are arranged in one of two distinct ways.  When the two
  squares are adjacent, we say $v$ is \emph{along a side} and when
  they are not, $v$ is a \emph{corner} of a region.  See
  Figure~\ref{fig:vtypes}.
\end{defn}

In each local configuration, the boundaries of the represented regions
are locally convex and thus every region is (globally) convex.  When
analyzing more complicated triangle-square flats, it is useful to
introduce a color scheme to highlight additional features.

\begin{defn}[Colors]\label{def:colors}
  Let $F$ be a triangle-square flat and assume it has been oriented so
  that at least one of its edges is vertical or horizontal.  The
  restrictive nature of the angles involved means that for every edge
  $e$ in $F$, the line $\ell$ containing $e$ points in one of six
  directions.  Concretely, the angle from horizontal is a multiple of
  $\frac{\pi}{6}$.  We divide the edges of $F$ into two classes based
  on the parallelism class of the lines containing them.  If the line
  containing $e$ is horizontal (or forms a $\frac{\pi}{3}$ angle with
  a horizontal line) then it is a \emph{blue edge}.  If the line
  containing $e$ is vertical (or forms a $\frac{\pi}{3}$ angle with a
  vertical line) then it is a \emph{red edge}.  Under these
  definitions, the six possible parallelism classes alternate between
  red and blue.  Note that all three edges of a triangle receive the
  same color so that it makes sense to speak of \emph{red triangles}
  and \emph{blue triangles}.  The edges of a square, on the other
  hand, alternate in color.  Being neither red nor blue, squares are
  assigned a third color such as yellow.  (For those viewing an
  electronic version of this article, we have followed these coloring
  conventions in all of our illustrations.)  When $F$ is oriented so
  that our color conventions apply, we say that $F$ is \emph{colored}.
\end{defn}

\begin{figure}
  \begin{center}
    \begin{tabular}{ccc}
      \includegraphics[scale=0.45]{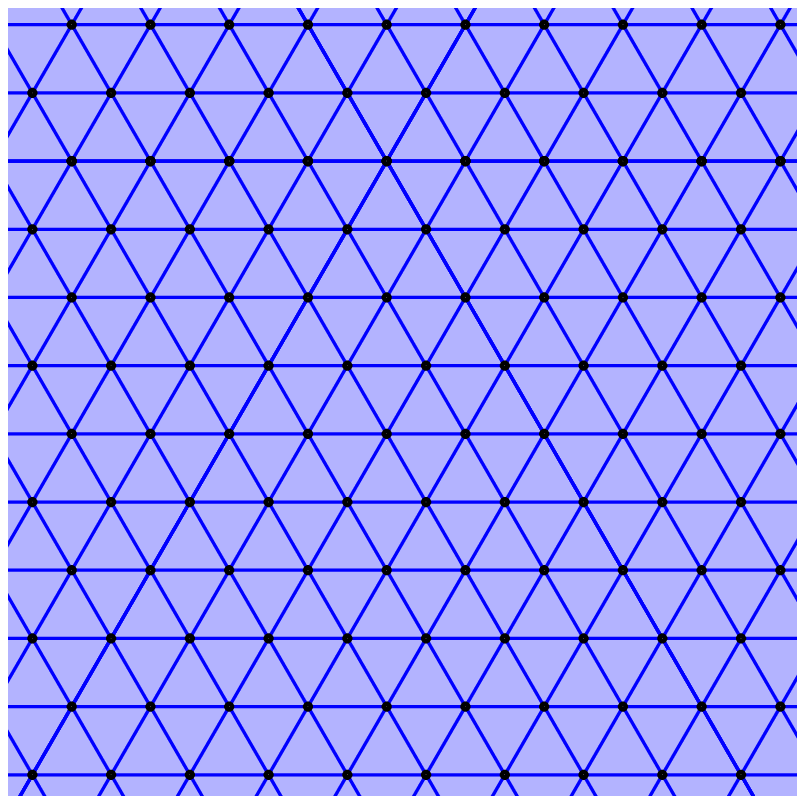} & 
      \includegraphics[scale=0.45]{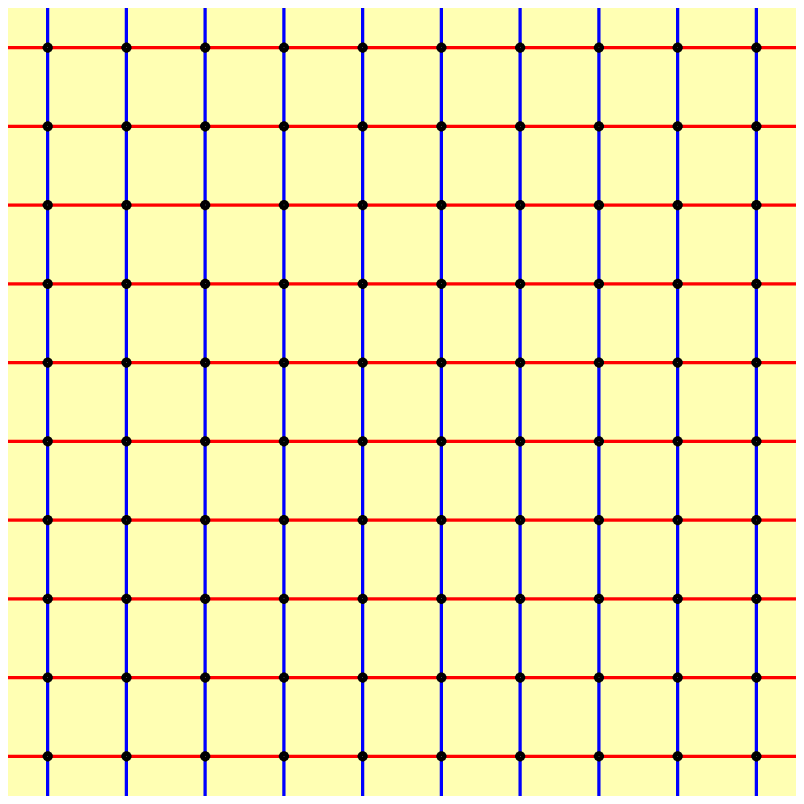} & 
      \includegraphics[scale=0.45]{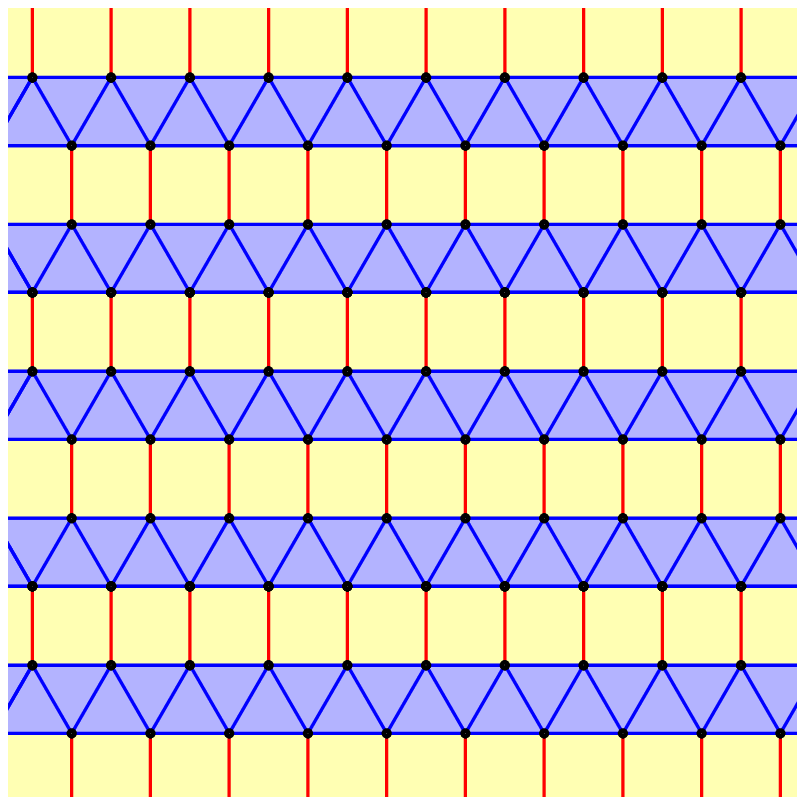}\\
      \includegraphics[scale=0.45]{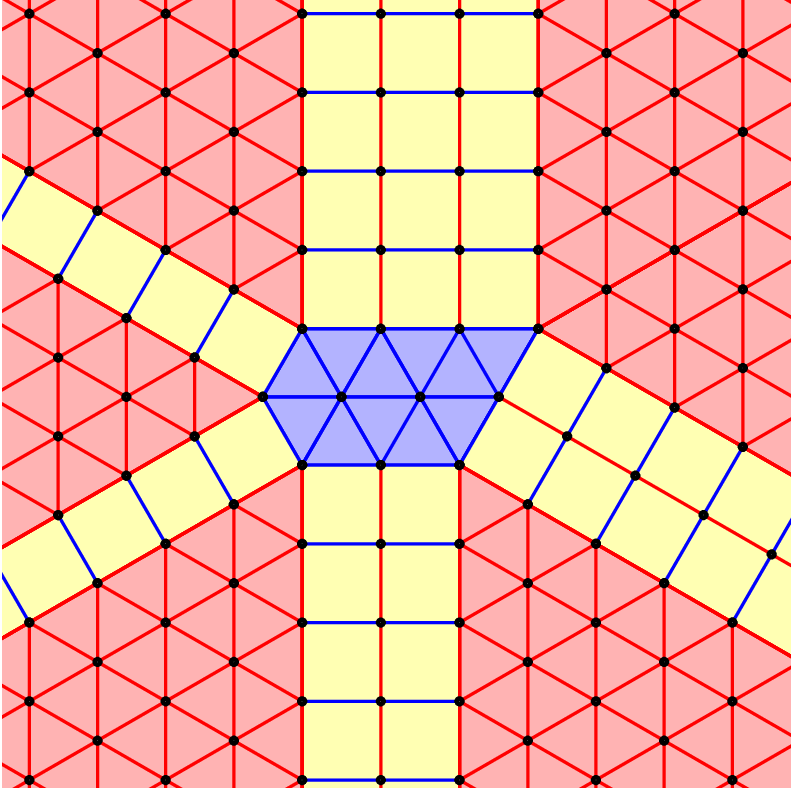} & 
      \includegraphics[scale=0.45]{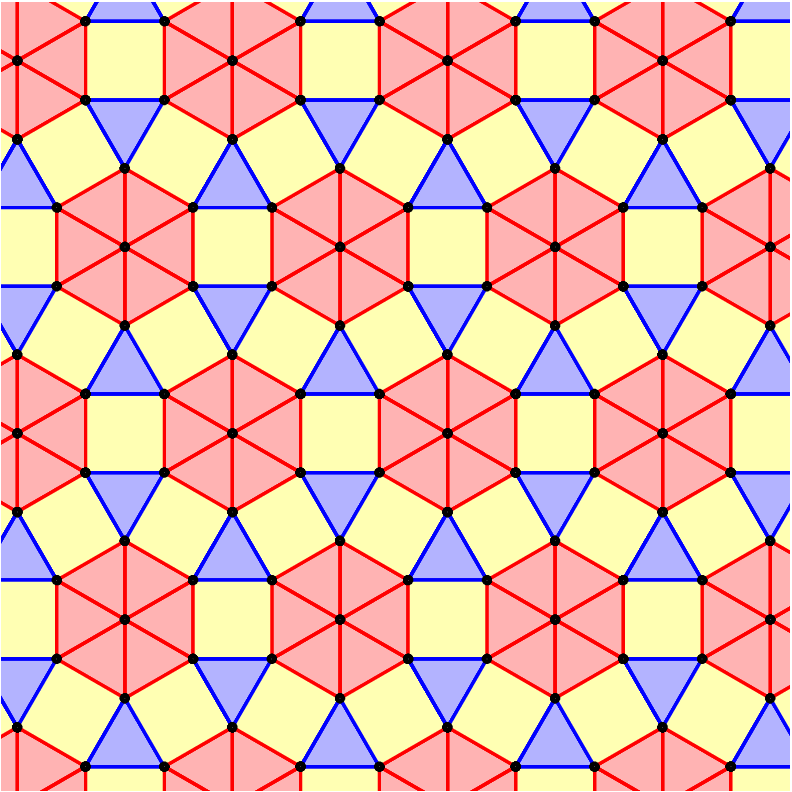} & 
      \includegraphics[scale=0.45]{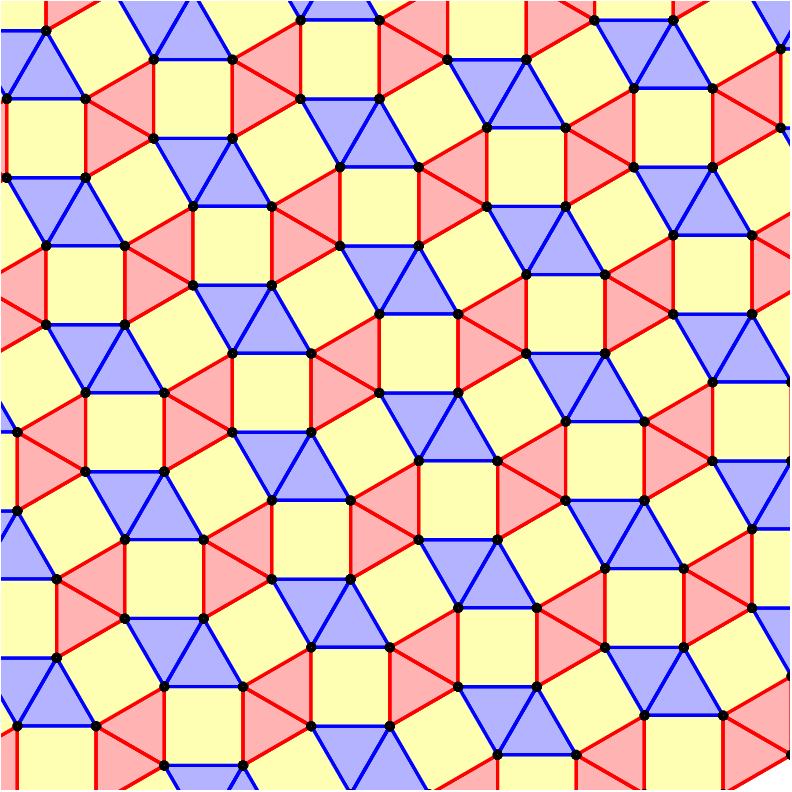}\\
    \end{tabular}
  \end{center}
  \caption{Six examples of triangle-square flats. They are pure
    triangle, pure square, striped, radial, and a pair of thoroughly
    crumpled flats.  The radial flat is intrinsically aperioidic; the
    other five are potentially periodic.\label{fig:flats}}
\end{figure}

\begin{defn}[Types of flats]
  Let $F$ be a triangle-square flat.  If $F$ contains only triangles
  or only squares, we say that $F$ is \emph{pure}.  If $F$ contains
  both triangles and squares but it contains no corner vertices, then
  $F$ is \emph{striped}.  If $F$ contains a corner vertex and at least
  one region that is unbounded, then $F$ is \emph{radial}.  Finally,
  if every region of $F$ is bounded, then $F$ is \emph{crumpled} and
  if there is a uniform bound on the number of $2$-cells in a region
  of $F$, we say $F$ is \emph{thoroughly crumpled}.  Several examples
  are shown in Figure \ref{fig:flats}).
\end{defn}

The reason that impure flats with no corners are called striped planes
is explained by the following observation.

\begin{lem}[Striped flats]\label{lem:striped}
  Let $F$ be a triangle-square flat.  If one region of $F$ has a
  straight line as part of its boundary then every region of $F$ is
  bounded by parallel straight lines.  In particular, a
  triangle-square flat is striped iff it contains a region with a
  straight line as part of its boundary, which is true iff every
  region is an infinite strip or a half-plane.
\end{lem}

\begin{proof}
  Let $R$ be the region with a straightline as part of its boundary.
  If this is all of its boundary then $R$ is a half-plane.  If it is
  not all of its boundary then the convexity of $R$ forces the
  remaining portion to be a parallel straightline (since corners lead
  to contradictions to convexity and nonparallel straight lines would
  intersect).  In other words, $R$ must be an infinite strip, an
  interval cross $\R$.  Shifting our the region on the other side of a
  boundary line of $R$ we see that it too is either a half-plane or a
  strip and we can continue in this way until the entire plane is
  exhausted.  The final assertion is now nearly immediate.  If $F$ has
  both triangles and squares but no corner vertices then there is an
  edge bordering both a triangle and a square and it must extend to a
  straight line boundary between regions.  As argued above, it follows
  that every region is then an infinite strip or a half-plane.
  Conversely, when every region is either a half-plane or an infinite
  strip, the flat clearly contains both triangles and squares and no
  corner vertices.
\end{proof}

As a corollary of Lemma~\ref{lem:striped} note that the obvious flats
embedded in a triplane must be pure or striped since triplanes are
formed from three tiled half-planes with a common boundary.  The wide
variety of cell structures exhibited by triangle-square flats is the
key difficulty encountered when trying to extend results from pure
triangle and pure square complexes to mixed triangle-square complexes.
This is in sharp contrast with the pure situation since every flat in
a triangle complex or in a square complex is obviously pure.

\begin{defn}[Pure flats]\label{def:pure-flats}
  A pure triangle-square flat of either type is essentially unique in
  that it must look like the standard tiling of $\R^2$ by triangles or
  a standard tiling by squares.  Because the vertex sets of pure flats
  can be identified with the Eisenstein integers $\Z[\omega]$ in the
  triangle case or the Gaussian integers $\Z[i]$ in the square case
  (where $\omega$ and $i$ are primitive third and fourth roots of
  unity), we call these complexes \emph{Eisenstein} and
  \emph{Gaussian} planes, respectively, and denote them $\E$ and $\G$.
\end{defn}

One illustration of the variety available in the mixed case, is the
existence of flats that are intrinsically aperiodic.

\begin{defn}[Intrinsically aperiodic flats]
  Let $K$ be a complex and let $F$ be a flat that embeds into its
  universal cover $\widetilde K$.  Composition with the covering
  projection immerses $F$ into $K$ itself and it is traditional to
  call $F$ \emph{periodic} or \emph{aperiodic} depending on whether or
  not this immersion $F \to K$ factors through metric Euclidean torus.
  This is equivalent to finding a subgroup isomorphic to $\Z \times
  \Z$ in $\pi_1(K)$ so that the corresponding deck transformations of
  $\widetilde K$ stabilize the image of $F$ setwise and act
  cocompactly on this image.  For our purposes we need a different
  distinction that is intrinsic to $F$ itself.  We say that $F$ is
  \emph{potentially periodic} if there exists a cocompact $\Z \times
  \Z$ action on $F$ by cellular maps and \emph{intrinsically
    aperiodic} if no such action exists.  We should note that by
  passing to a finite-index subgroup if necessary, we can always
  insist that the $\Z\times \Z$ action on $F$ consists solely of
  translations.
\end{defn}

Only certain types of flats can be potentially periodic.

\begin{prop}[Potentially periodic flats]\label{prop:aperiodic-flats}
  Every potentially periodic triangle-square flat is pure, striped or
  thoroughly crumpled.
\end{prop}

\begin{proof}
  Let $F$ be a potentially periodic triangle-square flat.  By
  assumption there is a cocompact action of $\Z\times \Z$ on $F$.  Let
  $n$ be the number of vertices in a fundamental domain for this
  action and note that regions are sent to regions under this cellular
  action.  If $F$ is not thoroughly crumpled, it contains a region $R$
  with more than $n$ vertices, which implies that there is a
  non-trivial translation of $F$ that sends $R$ to $R$.  In
  particular, $R$ is a convex subset of $\R^2$ that is invariant under
  a non-trivial translation.  Unless $F$ is pure, this means that $R$
  contains a straightline as part of its boundary and by
  Lemma~\ref{lem:striped}, the flat $F$ is striped.
\end{proof}

\section{Geodesics in flats}\label{sec:flat-geos}

As a first step towards showing that the Gersten-Short geodesics can
be used to define biautomatic structures, consider the following
question.  Suppose that $F$ is a triangle-square flat viewed as a
$\cat(0)$ triangle-square complex in its own right.  Does there exists
a universal constant $k$, possibly depending on the structure of $F$,
such that every pair of Gersten-Short paths in $F$ that start and end
at most $1$ unit apart synchronously $k$-fellow travel?  The answer,
unfortunately, is that this is true for certain types of flats (such
as pure flats and striped flats) but it is definitely not true in
general.  In other words, this is an instance where new phenomena are
produced when triangles and squares are mixed.  We begin by showing
that Gersten-Short paths in pure flats and striped flats are
well-behaved.

\begin{prop}[Fellow traveling in pure flats]\label{prop:ft-pure}
  If $F$ is a pure triangle or pure square flat, then pairs of
  Gersten-Short geodesics in $F$ that start and end within $1$ unit of
  each other synchronously $k$-fellow travel for a small explicit
  value of $k$.
\end{prop}

\begin{proof} 
  Let $F$ be a pure square flat and let $\alpha$ be a Gersten-Short
  geodesic from $u$ to $v$.  Because the structure of $F$ is so
  simple, we know that the choke points initially jump across a
  (possibly empty) sequence of square moves, followed by a (possibly
  empty) sequence of unique edges.  In particular, if we arrange $F$
  so that its edges are horizontal/vertical and then connect the
  successive choke points with straight segments, the result is a pair
  of line segments approximating $\alpha$ where the first has slope
  $\pm 1$ and the second is either horizontal or vertical.  See Figure
  \ref{fig:gs-pure}.  If $\beta$ is a second Gersten-Short geodesic
  that starts within $1$ unit of $u$ and ends within $1$ unit of $v$,
  the same procedure will produce a pair of line segments with the
  same slopes and in the same configuration (unless one of the two
  original segments is very short).  In particular, every choke point
  of $\alpha$ is connected to a choke point of $\beta$ by an edge and
  vice versa.  Since every vertex in $\alpha$ or $\beta$ is within $1$
  unit of a choke point, the subspace distance between the two paths
  is at most $3$ and by Proposition~\ref{prop:ft-geo}, $\alpha$ and
  $\beta$ synchronously $7$-fellow travel.  (With a bit more work one
  can show that $\alpha$ and $\beta$ synchronously $3$-fellow travel,
  but it is the existence of a bound that is significant.)  The
  argument for Gersten-Short geodesics in pure triangle flats is
  essentially identical.
\end{proof}

A similar result can be established for geodesics in striped flats.

\begin{figure}
  \begin{center}
    \begin{tabular}{ccc}
      \begin{tabular}{c}\includegraphics[scale=0.5]{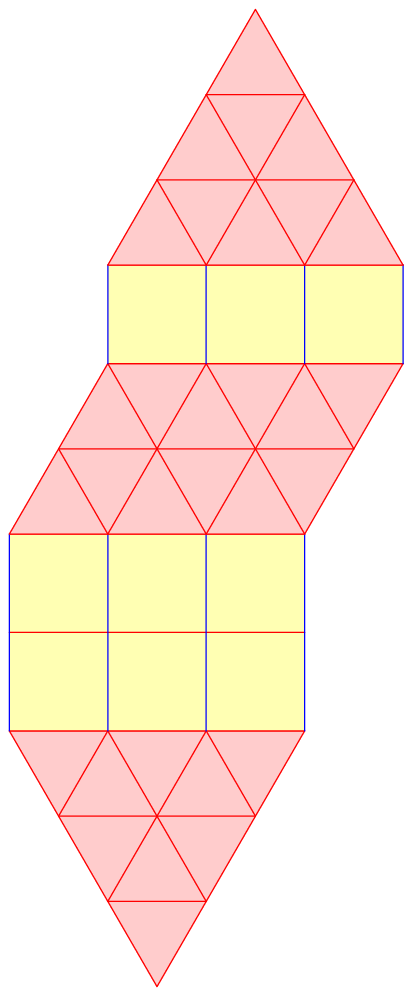}\end{tabular} & 
      \begin{tabular}{c}\includegraphics[scale=0.5]{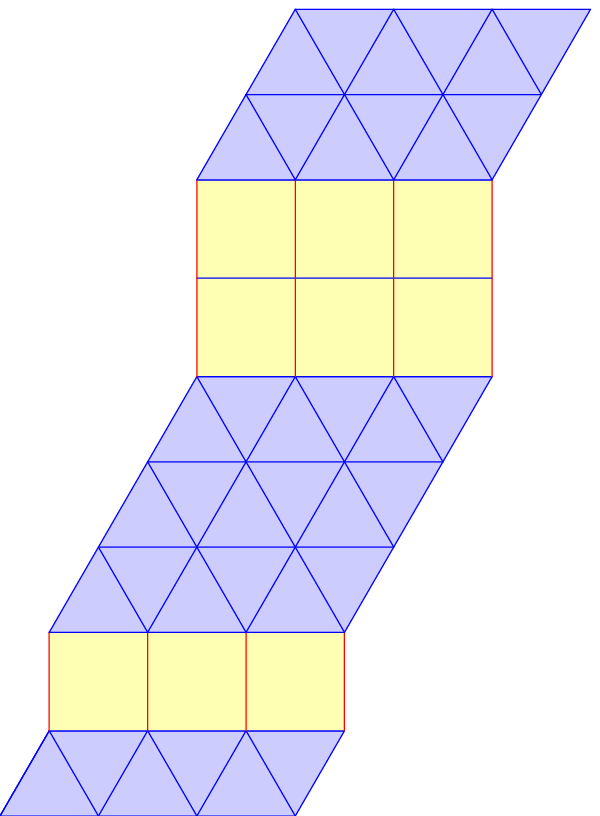}\end{tabular} &
      \begin{tabular}{c}\includegraphics[scale=0.5]{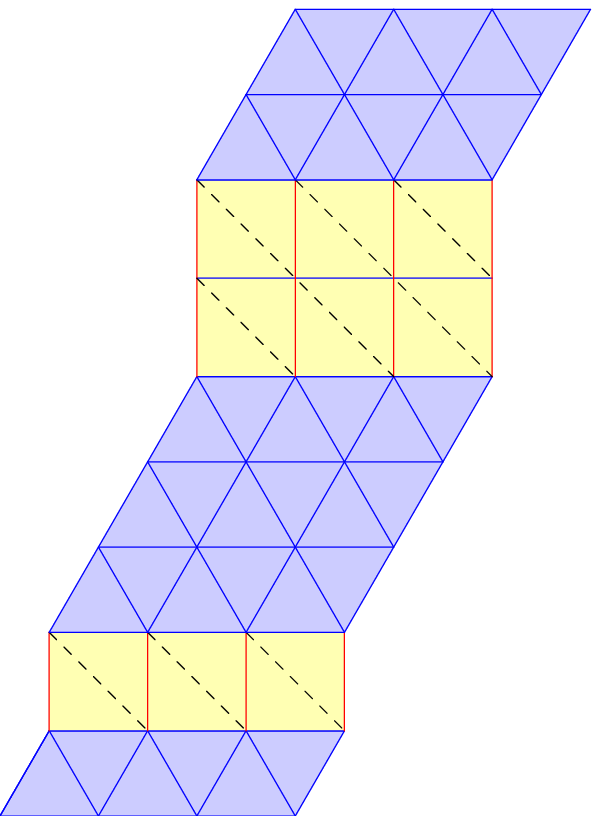}\end{tabular}\\
    \end{tabular}
  \end{center}
  \caption{Two examples of intervals in striped flats and a
    subdivision of the squares in the second example as described in
    the proof of Proposition
    \ref{prop:ft-striped}.\label{fig:ft-striped}}
\end{figure}

\begin{prop}[Fellow traveling in striped flats]\label{prop:ft-striped} 
  If $F$ is a striped flat, then pairs of Gersten-Short geodesics in
  $F$ that start and end within $1$ unit of each other synchronously
  $k$-fellow travel for a small explicit value of $k$.
\end{prop}

\begin{proof} 
  Let $F[u,v]$ be an interval in a striped flat $F$ between vertices
  $u$ and $v$.  If $F$ is arranged so that the boundaries between
  regions are horizontal, then $F[u,v]$ can viewed as an interval in a
  pure triangle flat, i.e. a (possibly degenerate) triangulated
  parallelogram, with with horizontal square strips inserted.  Two
  possibilities are illustrated in Figure~\ref{fig:ft-striped}.  The
  key difference between the two examples is whether the edges that
  lie on the boundary between a triangular region and a square region
  are ``horizontal'' or ``vertical'' in the terminology of
  Definition~\ref{def:v/h-edges}.  When one of these boundary edges is
  ``horizontal'', they are all ``horizontal'' and every square in a
  square region of $F[u,v]$ is part of a triangle-square-triangle move
  where the line segment connecting the choke points at either end is
  vertical.  In particular, these line segments are all parallel to
  each other and the distance between the Gersten-Short paths in this
  interval is the same as the distance between Gersten-Short paths in
  the pure triangle regions formed by vertically collapsing every
  square simultaneously.  On the other hand, when one of these
  boundary edges is ``vertical'', they are all ``vertical'' and every
  square in a square region of $F[u,v]$ is part of a square move.  In
  particular, diagonal lines can be added connecting the opposite
  vertices of each square that are the same distance from $u$.  The
  result is a new combinatorial pattern that looks like a portion of a
  pure triangle flat (as shown on the right hand side of
  Figure~\ref{fig:ft-striped}).  The key observation is that the set
  of paths that are combinatorial geodesics did not change.  The
  arguments used to prove Proposition~\ref{prop:ft-pure} show that
  Gersten-Short geodesics in intervals of this type that start and end
  close to each other remain close throughout.  The only change is
  that as a final step, the inserted diagonals need to be removed
  which merely forces the old $k$ to be replaced by a slightly larger
  $k'$.
\end{proof}

Unfortunately, concrete examples show that these types of arguments
cannot be extended to cover all triangle-square flats.

\begin{figure}
  \includegraphics[scale=.9]{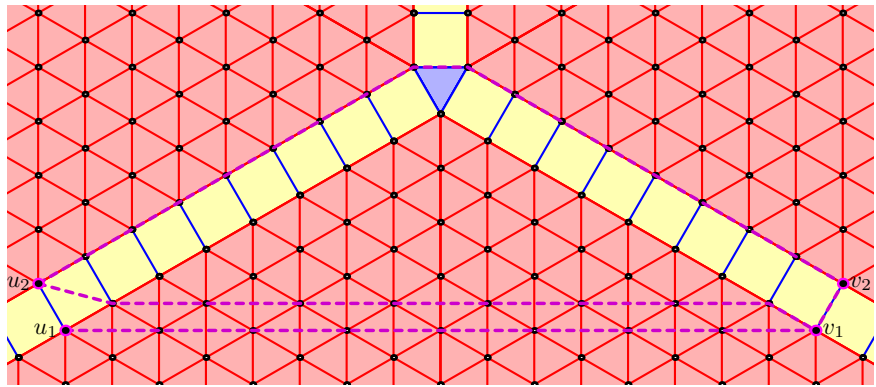}
  \caption{Paths that fail to $2$-fellow-travel.\label{fig:bad-plane}}
\end{figure}

\begin{prop}[Fellow traveling in radial flats]\label{prop:ft-radial} 
  There exist a radial flat plane $F$ such that for every $k>0$ we can
  find pairs of Gersten-Short geodesics $\alpha$ and $\beta$ in $F$
  that start and end an edge apart but where the subspace distance
  between $\alpha$ and $\beta$ is larger than $k$.  In particular,
  there are radial flat planes in which the Gersten-Short geodesics do
  not synchronously $k$-fellow travel for any value of $k$.
\end{prop}

\begin{proof}
  Let $F$ be the radial flat plane shown in Figure~\ref{fig:bad-plane}
  with only one bounded region consisting of a single triangle and all
  other regions unbounded.  Next, let $u_1$ and $u_2$ be two vertices
  connected by an edge in the interior of one of the three square
  regions distance $\ell$ from the bounded region.  Similarly let
  $v_1$ and $v_2$ be two vertices connected by an edge in the interior
  of a different square region distance $\ell$ from the one bounded
  region with subscripts chosen so that $u_1$ and $v_1$ belonging to
  the boundary of the same triangular region.  The figure illustrates
  such a configuration for $\ell = 8$.  The Gersten-Short geodesics in
  the intervals $F[u_1,v_1]$, in $F[u_2,v_1]$ and in $F[u_1,v_2]$ are
  all roughly parallel (and, in this illustration, roughly
  horizontal).  In $F[u_2,v_2]$, however, the only geodesic, and
  therefore the unique Gersten-Short geodesic, is an edge path that
  travels along the boundaries of the square regions and includes an
  edge from the central triangle.  By chosing $\ell$ sufficiently
  large, the subspace distance between any Gersten-Short geodesic in
  $F[u_1,v_1]$ and the Gersten-Short geodesic in $F[u_2,v_2]$ can be
  made as large as possible.  In particular, the Gersten-Short
  geodesics in $F$ do not synchronously $k$-fellow travel for any
  value of $k$.
\end{proof}

At this point it is important to emphasize what this example does and
does not show.  Although it demonstrates that there are problems that
need to be avoided, it does not in and of itself show that
Gersten-Short geodesics cannot be used to establish the main
conjecture, Conjecture~\ref{conj:biaut}.  This is because the radial
flat used as a counterexample is intrinsically aperiodic
(Proposition~\ref{prop:aperiodic-flats}) and it is not at all clear
that such a flat can occur in the universal cover of a compact
nonpositively curved triangle-square complex $K$.  In fact, we
conjecture that no such embeddings exist.

\begin{conj}[Intrinsically aperiodic flats]\label{conj:aperiodic}
  If $K$ is a compact nonpositively curved triangle-square complex and
  $F$ is an intrinsically aperiodic flat then $F$ does not embed into
  $\widetilde K$ and it does not immerse into $K$.
\end{conj}

On the other hand, since arbitrarily large portions of this example
can occur in potentially periodic flats, this example does show that
it is not possible to establish a global value $k_0$ such that pairs
of Gersten-Short geodesics that start and end within one unit of each
other in the universal cover of \emph{any} compact nonpositively
curved triangle-square complex $K$ sychronously $k_0$-fellow travel.
This is in contrast with the pure cases studied by Gersten and Short
where such a global value does indeed exist.

\section{Geodesics in periodic flats}\label{sec:periodic-flats}

At this point we have shown that Gersten-Short geodesics are
well-behaved in pure flats and striped flats but they need not fellow
travel in radial flats.  More over, we have conjectured
(Conjecture~\ref{conj:aperiodic}) that the only flats that occur in
the universal cover of a compact nonpositively curved triangle-square
complex are ones that are potentially periodic.  In this section we
prove that Gersten-Short geodesics are well-behaved in every
potentially periodic flat $F$ by extending the earlier results for
pure flats and striped flats to flats that are thoroughly crumpled
(Proposition~\ref{prop:aperiodic-flats}).  The argument is of
necessity more delicate since the fellow traveling constant $k$ must
now vary depending on the fine structure of the flat $F$ under
consideration.  We begin by imposing order on the wide variety of
triangle-square flats in a slightly surprising fashion.  In
particular, we prove that every triangle-square flat embeds in a
single $4$-dimensional \pe complex, the product $\E \times \E$ of two
Eisenstein planes.  This is a tiling of $\R^4$ by $4$-polytopes that
are the metric product of two equilateral triangles.

\begin{defn}[Projections]
  Let $F$ be a triangle-square flat arranged so that it is colored
  (Definition~\ref{def:colors}) and imagine altering the metric on $F$
  as follows.  Let every red edge have length $s$, let every blue edge
  length $t$ and keep all of the angles the same.  Under this new
  metric the red and blue triangles are rescaled equilateral triangles
  and the yellow squares become $s$ by $t$ rectangles.  For every
  choice of positive reals $s$ and $t$ the resulting metric space
  remains a flat, i.e. isometric to $\R^2$.  Next, consider the
  limiting case where $s$ shrinks to $0$ while $t$ is fixed at $1$.
  The end result is (a subcomplex of) the Eisenstein tiling $\E$ and
  it is the same as the quotient of $F$ that collapses red edges and
  red triangles to points, and collapses every yellow square to a blue
  edge.  Visually, what happens is that the blue triangles in $F$
  slide together.  See Figure~\ref{fig:projections}.  Similarly,
  reversing the roles the colors play slides the red triangles
  together to produce a second projection from $F$ to (a rotated
  version of) an Eisenstein tiling $\E$.  We call these projections
  $p_b: F \longrightarrow \E$ and $p_r: F \longrightarrow \E$ where
  the subscript indicates the color of the edges that survive.
  Although there are minor ambiguities in the definition of these maps
  (such as which vertex is sent to the origin and which edges have
  horizontal images), they are not substantial since any two possible
  definitions of $p_b$ or $p_r$ differ by a cellular isometry $\E \to \E$.
\end{defn}

\begin{figure}
  \begin{center}
    \begin{tabular}{ccc}
      \includegraphics[scale=0.45]{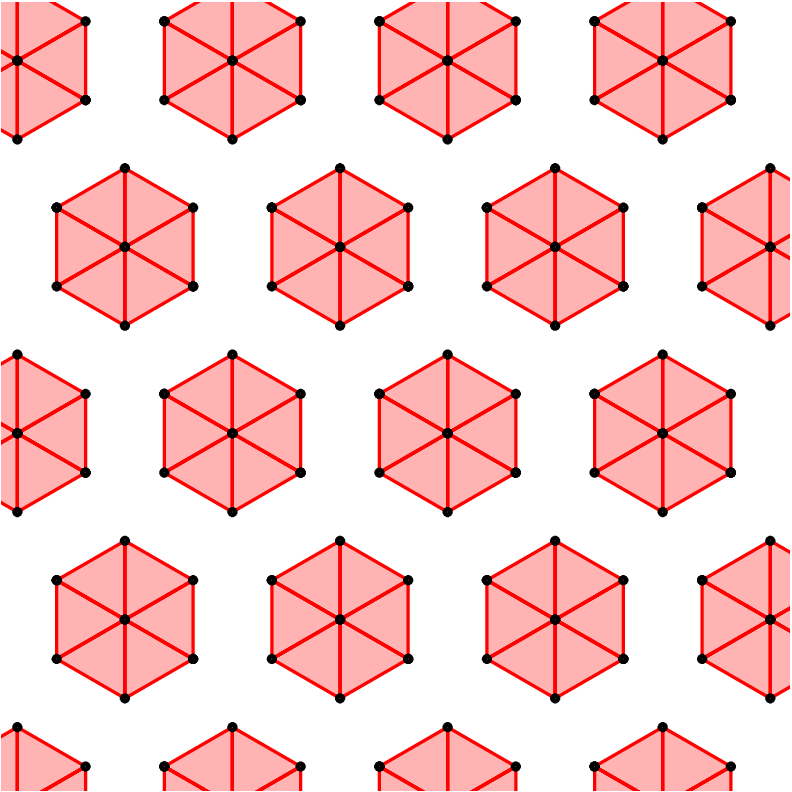} &
      \includegraphics[scale=0.45]{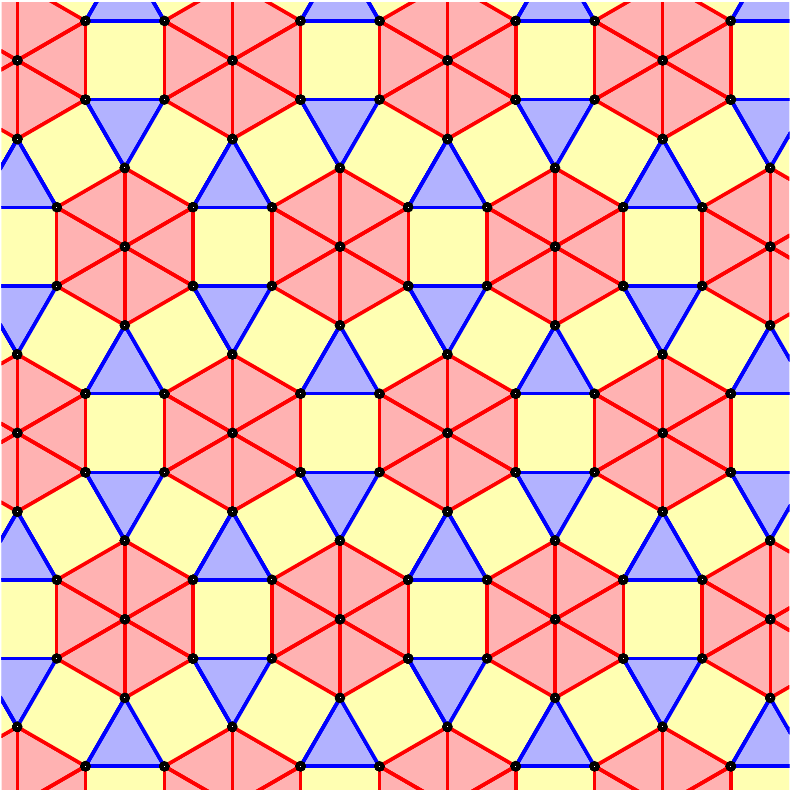} & 
      \includegraphics[scale=0.45]{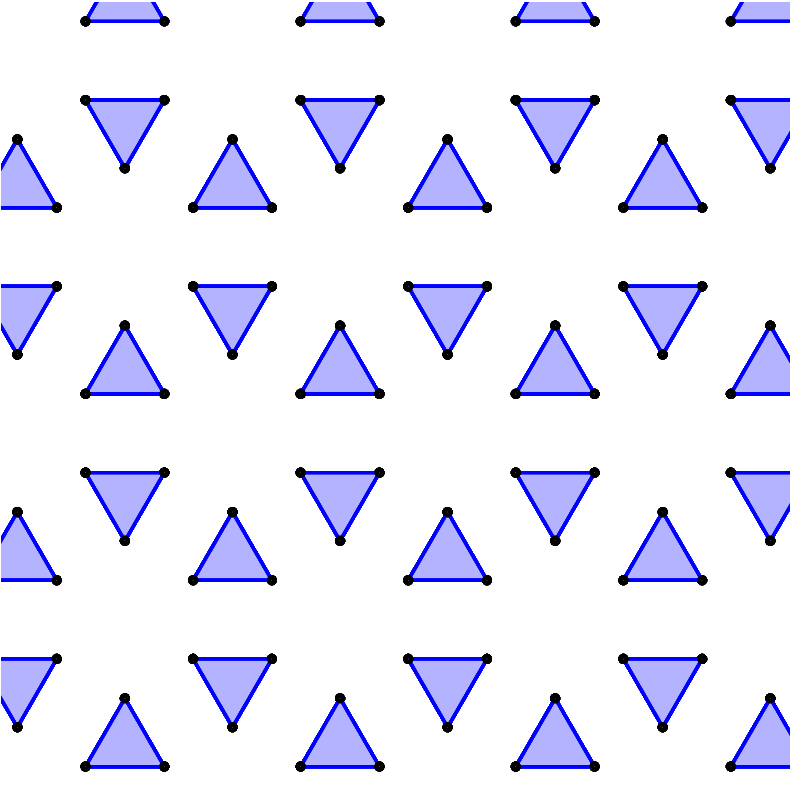}
    \end{tabular}
  \end{center}
  \caption{A periodic flat with its ``red'' and ``blue'' triangles
    shown on either side.\label{fig:projections}}
\end{figure}

Locally it is clear that adjacent vertices are identified under one of
these projections iff the edge connecting them is one of those that
collapses and two sides of a $2$-cell are identified iff they are
opposite sides of a square and of the non-collapsing color.  The
following proposition is the global version of these observations.

\begin{prop}[Identifications]\label{prop:identifications}
  Let $F$ be a colored triangle-square flat and let $p_r$ and $p_b$ be
  the red and blue projections defined above. If $u$ and $v$ are
  vertices of $F$, then $p_r(u) = p_r(v)$ iff there is a pure blue
  path in $F$ from $u$ to $v$ and $p_b(u) = p_b(v)$ iff there is a
  pure red path in $F$ from $u$ to $v$.  More generally, the preimage
  of vertex in $\E$ is a (possibly degenerate) triangular region and
  the preimage of an edge is a (possibly degenerate) strip of squares.
\end{prop}

As a corollary we note that crumpled plane projections are onto.

\begin{cor}[Projections onto $\E$]\label{cor:project-onto}
  If $F$ is a colored crumpled triangle-square flat then the
  projections $p_r$ and $p_b$ are onto maps.
\end{cor}

\begin{proof}
  Let $e$ be an edge in the image of $F$ under one of the projection
  maps.  By Proposition~\ref{prop:identifications} its preimage is a
  strip of squares and because $F$ is crumpled, this strip is finite.
  The triangles on either end of this strip show that $e$ is in the
  interior of the image of $F$.  Thus the image has no boundary edges
  and must therefore be all of $\E$.
\end{proof}

A more precise statement would be that both projections are onto iff
every square region of $F$ is bounded.  We are now ready to show that
every triangle-square flat, no matter how complicated, can be embedded
in a direct product of two Eisenstein planes.

\begin{thm}[Embedding flats into $\E \times \E$] \label{thm:flat-embedding}
  Let $F$ be a colored triangle-square flat and let $p_b$ and $p_r$ be
  the blue and red projections defined above.  If $p: F\to \E \times
  \E$ is the map sending $x$ to the ordered pair $(p_b(x),p_r(x))$,
  then $p$ is a piecewise linear cellular map that embeds $F$ into the
  $2$-skeleton of the $4$-dimensional complex $\E \times \E$ built out
  of copies of the $4$-polytope formed as a direct product of two
  equilateral triangles.
\end{thm}

\begin{proof}
  It is easy to check that vertices, edges, triangles and squares in
  $F$ are sent to vertices, edges, triangles and squares in
  $\E\times\E$.  Thus the only assertion that needs to be verified is
  that $p$ is an embedding.  In fact, it is sufficient to show that
  $p$ is injective on the vertices of $F$.  By
  Proposition~\ref{prop:identifications}, if $u$ and $v$ are distinct
  vertices of $F$ with $p_r(u)=p_r(v)$, then there is a combinatorial
  path in $F$ connecting them that consists solely of blue edges.  But
  this blue path remains rigid and unchanged under the projection
  $p_b$.  In other words, when $p(u)$ and $p(v)$ agree in the second
  coordinate, their first coordinates disagree.  Repeating this
  argument with the colors reversed shows that distinct vertices are
  sent to vertices that might agree on one coordinate but cannot agree
  on both.
\end{proof}

In fact, the embedding of $F$ into $\E\times \E$ satisfies an even
stronger condition of being an isometric embedding with respect to the
$1$-skeleton metric.  Although true for arbitrary flats, we restrict
our attention to flats that are crumpled.

\begin{thm}[Embedding flats isometrically]\label{thm:isom-embed}
  If $F$ be a colored crumpled triangle-square flat and $p:F\to
  \E\times\E$ is the embedding described above, then $p$ preserves
  combinatorial distances between vertices.  In particular, $p$
  isometrically embeds the $1$-skeleton of $F$ into the $1$-skeleton
  of $\E\times \E$.
\end{thm}

\begin{proof}
  Because $p$ is a cellular map, the distance between vertices in $F$
  is at least as large as the distance between their images in $\E
  \times \E$.  Thus, if the theorem fails, there is a geodesic
  $\alpha$ in $F$ whose image is not a geodesic in $\E\times \E$ and
  thus not a geodesic under one of the projection maps.  Arguing by
  contradiction, let $\alpha$ a minimal such path and assume, without
  loss of generality, that its image is not geodesic under the
  projection $p_b$.  The only possibilities are that $p_b(\alpha)$ is
  the old path of a trivial move or it looks like the old path in a
  doubly-based diagram similar to the one shown on the right in
  Figure~\ref{fig:e-shortening}, i.e. one that can be shortened using
  a sequence of triangle-triangle moves followed by a triangle move,
  (and this includes the old path of triangle move as a degenerate
  case).  That these are the only two possibilities is a relatively
  straightforward consequence of applying
  Theorem~\ref{thm:straightening} to paths in $\E$.  When the image is
  the old path of a trivial move, the preimage of the edge traversed
  is a finite strip of squares in $F$ and $\alpha$ traverses three
  sides of this rectangle.  The remaining side shows that $\alpha$ was
  not a geodesic in $F$, contradiction.  When the projection of
  $\alpha$ looks like the old path of the diagram on the right in
  Figure~\ref{fig:e-shortening} we consider the preimage of the
  diagram $D$ bounded by $p_b(\alpha)$ and the straight line geodesic
  path that connects its two endpoints.  For concreteness, assume that
  this geodesic is horizontal.  By Corollary~\ref{cor:project-onto}
  the projection $p_b$ is onto and by
  Propostion~\ref{prop:identifications} the preimage contains a
  portion that looks like the diagram on the left in
  Figure~\ref{fig:e-shortening} (i.e. triangles connected by finite
  strips of squares).  Since the path $\beta$ along the bottom of this
  diagram consists solely of edges that are either horizontal or an
  angle of $\frac{\pi}{6}$ from horizontal, its projections are
  geodesics which implies that it is sent to a geodesic in $\E \times
  \E$ and thus a geodesic in $F$.  Moreover, $\beta$ is strictly
  shorter than $\alpha$, since its $p_b$ projection is strictly
  shorter and its $p_r$ projection is a geodesic.  But this means that
  $\alpha$ is not a geodesic in $F$, contradiction.
\end{proof}

\begin{figure}
  \begin{center}
    \begin{tabular}{cc}
      \begin{tabular}{c}\includegraphics[scale=.8]{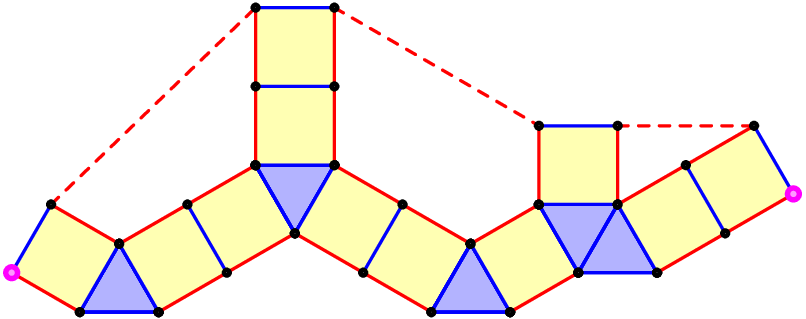}\end{tabular} &
      \begin{tabular}{c}\includegraphics[scale=.8]{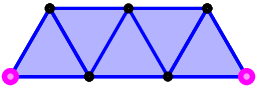}\end{tabular}
    \end{tabular}
  \end{center}
  \caption{The diagram on the right is a doubly-based disc diagram
    whose old path can be shortened using two triangle-triangle moves
    and one triangle move.  The diagram on the left shows a portion of
    a possible preimage of this diagram in a triangle-square flat $F$
    with respect to the projection map $p_b$.  The dashed red lines
    represent paths in the missing red triangular
    regions.\label{fig:e-shortening}}
\end{figure}

By Theorem~\ref{thm:flat-embedding} geodesics in the image of a flat
correspond exactly with geodesics in the domain, making the following
corollary immediate.

\begin{cor}[Images of intervals]\label{cor:interval-images}
  Let $F$ be a colored crumpled triangle-square flat and let $p$ be
  its embedding into $K=\E\times \E$.  For all vertices $u$ and $v$ in
  $F$ with images $u'$ and $v'$ in $K$, the image of the interval
  $F[u,v]$ is the intersection of the interval $K[u',v']$ with the
  embedded flat $p(F)$.
\end{cor}

This corollary is interesting, at least in part, because intervals in
$\E \times \E$ are extremely simple in form.

\begin{prop}[Intervals in $\E \times \E$]
  For all vertices $u$ and $v$ in the $4$-complex $K = \E \times \E$,
  the interval $K[u,v]$ is a (possibly degenerate) $4$-dimensional
  parallelepiped.
\end{prop}

\begin{proof}
  Because $K$ is a direct product, a vertex $w$ is between $u$ and $v$
  iff the image of $w$ is between the images of $u$ and $v$ under both
  projections $p_r$ and $p_b$.  In particular, the interval in $K$ is
  a direct product of two intervals in $\E$ and these are (possibly
  degenerate) $2$-dimensional parallelograms.
\end{proof}

With these tools in hand, we now shift our attention to flats that are
potentially periodic.  The self-similarity of the potentially
perioidic flat means that its image under the embedding described
above looks, roughly speaking, like an affine plane in $\R^4$, a
description that can be made precise using the notion of a
quasi-isometry.

\begin{defn}[Quasi-isometries] 
  Let $f:X\to Y$ be a function between metric spaces.  If there are
  constants $\lambda$ and $\delta$ such that for all $x,x' \in X$ with
  images $y, y' \in Y$ both $d_Y(y,y') \leq \lambda \cdot d_X(x,x') +
  \delta$ and $d_X(x,x') \leq \lambda \cdot d_Y(y,y') + \delta$, then
  $f$ is called a \emph{quasi-isometric embedding} of $X$ into $Y$. If
  $f$ is also \emph{quasi-onto} (i.e. there is a constant $c$ such
  that the $c$-neighborhood of $f(X)$ is all of $Y$) then $f$ is
  called a \emph{quasi-isometry} and $X$ and $Y$ are
  \emph{quasi-isometric}.
\end{defn}

Although it is not immediately obvious from this definition, the
relation of being quasi-isometric in an equivalence relation on metric
spaces.

\begin{prop}[Quasi-flats]\label{prop:quasi-flat}
  Let $F$ be a colored triangle-square flat and let $p:F\to \E \times
  \E$ be the corresponding embedding.  If $F$ is potentially periodic
  then its image $p(F)$ is quasi-isometric to a flat affine plane in
  $\R^4 \cong \E \times \E$.
\end{prop}

\begin{proof} 
  Let $u$ be a vertex of $F$ embedded in $\E \times \E \cong
  \R^{4}$. Because $F$ is potentially periodic there are a pair of
  linearly independent translations from $F$ to itself that generate a
  cocompact action of $\Z\times \Z$ on $F$.  Let $v$ and $w$ be the
  images of $u$ under these two maps.  Since the translations acting
  on $F$ do not change colors or orientations, they also induce
  translations of $\E \times \E$ that preserve the image of $F$.  In
  particular, if we let $A$ be the affine span in $\R^4 = \E \times
  \E$ of the points $p(u)$, $p(v)$ and $p(w)$, then each point of
  $p(F)$ lies within a bounded distance of a point in $A$ (since each
  point $F$ lies within a bounded distance of a point in the orbit of
  $u$ under the $\Z \times \Z$ action).  Thus, by choosing $c$ large
  enough, $p(F)$ lies in a $c$-neighborhood of $A$ which implies that
  $F$ is quasi-isometric to $A$.  Finally, note that $p(u)$, $p(v)$
  and $p(w)$ cannot be collinear since Euclidean planes and lines are
  not quasi-isometric. Thus $A$ is a $2$-dimensional affine plane.
\end{proof}

And finally, in order to understand the behavior of Gersten-Short
geodesic in crumpled triangle-square flats, we need to know that the
direction of the next choke point can be determined locally.

\begin{lem}[Determining the first move]\label{lem:next-move}
  Let $F$ be a colored crumpled triangle-square flat, let $p$ be the
  corresponding embedding of $F$ into $K=\E \times \E$ and let $u$ and
  $v$ be vertices in $F$ with images $u'$ and $v'$ in $K$.  The first
  move of a Gersten-Short geodesic from $u$ to $v$ is completely
  determined by the vertex neighborhood of $u$.  More specifically,
  the edges touching $u$ whose images lie in the parallelopiped
  $K[u',v']$ span the initial cell of the first move.
\end{lem}

\begin{proof}
  By Theorem~\ref{thm:first-move} the first move is determined by its
  first cell which in turn is determined by the edges leaving $u$ that
  extend to geodesics connecting $u$ to $v$ and such edges are
  characterized by whether or not their images lie in $K[u',v']$.
  More specifically, if an edge leaving $u$ is the first edge of a
  geodesic from $u$ to $v$ then by Corollary~\ref{cor:interval-images}
  its image lies in $K[u',v']$.  Conversely, any edge leaving $u$
  whose image lies in $K[u',v']$ is contained in $K[u',v'] \cap p(F)$,
  by Corollary~\ref{cor:interval-images} it is contained in
  $p(F[u,v])$, and thus it extends to a geodesic from $u$ to $v$.
\end{proof}

As a consequence of local determination of the next move, we can show
that Gersten-Short geodesics in a flat do not cross in the following
sense.

\begin{cor}[Noncrossing geodesics]\label{cor:noncrossing} 
  Let $F$ be a colored crumpled triangle-square flat and let $p$ be
  the corresponding embedding of $F$ into $K=\E \times \E$. If $u$ and
  $v$ are vertices in $F$ with images $u'$ and $v'$ in $K$, then for
  every pair of vertices $x$ and $y$ within $1$ unit of $u$ and $v$ in
  $F$, the next choke point in the Gersten-Short geodesic from $x$ to
  $y$ is locally determined by the neighborhood of $x$ and the
  straight line segment from $x$ to this next choke point in $F$ does
  not cross the straight line segment from $u$ to its next choke
  point.
\end{cor}

\begin{proof}
  Since the parallelepipeds $K[u',v']$ and $K[x',y']$ are nearly
  identical, except in situations where one of them is degenerate, we
  can assume by Lemma~\ref{lem:next-move} that $v=y$.  Next, note that
  the edges crossed by the straight line segment from $u$ to its first
  choke point are edges with both endpoints the same distance from $v$
  and also note that no other edges in the move have this property.
  Because $v=y$ the edges crossed by the straight line segment from
  $x$ to its next choke point have the same property.  In particular,
  neither straight line segment can cross into the interior of the
  move containing the other.  The argument in the exceptional case is
  essentially the same.
\end{proof}

Combining Lemma~\ref{lem:next-move} and
Corollary~\ref{cor:noncrossing}, we now show that image of a
Gersten-Short geodesic in $\E \times \E$ is approximately linear until
the situation degenerates.

\begin{lem}[Roughly linear]\label{lem:quasi-linear} 
  Let $F$ be a colored potentially periodic flat, let $p$ be the
  corresponding embedding of $F$ into $K=\E \times \E$, and let $u$
  and $v$ be vertices in $F$ with images $u'$ and $v'$ in $K$.  There
  exists a direction vector and a constant depending only on the
  structure of $F$ and the directions of the sides of the
  parallelepiped $K[u',v']$ so that image under $p$ of any
  Gersten-Short geodesic from $u$ to $v$ stays with this constant of
  the line through $u'$ with this direction until that point $w'$ at
  which the interval $K[w',v']$ degenerates into a lower dimensional
  parallelepiped.
\end{lem}

\begin{proof} 
  The basic idea goes as follows.  Pick $u$ and $v$ so that $K[u',v']$
  is far from degenerate (i.e. all four side lengths are large).  By
  Lemma~\ref{lem:next-move} the first choke point only depends on the
  neighborhood of $u$ and the directions of the sides of $K[u',v']$.
  Repeat.  If the initial side lengths were large enough, then we pass
  through two vertices that lies in the same $\Z\times \Z$ orbit
  before the remaining interval is able to degenerate.  The vector
  connecting these two vertices is the vector mentioned in the
  theorem.  By Lemma~\ref{lem:next-move} again, the same sequence of
  moves is now repeated until that point at which the interval
  remaining degenerates.  This shows that this portion of the
  Gersten-Short geodesic approximates a line in $K$ to within a
  specified constant.  Finally, if we pick any other vertex $x$ close
  to $u$ and repeat the argument above, we find that the Gersten-Short
  geodesic from $x$ also has an image in $K$ that approximates a
  straight line.  Because Gersten-Short geodesics in $F$ do not cross
  (Corollary~\ref{cor:noncrossing}) the direction of this line must
  agree with the previous direction (since it can be bounded on both
  sides by lines originating in $u$ and other nearby vertices in the
  same $\Z \times \Z$ orbit as $u$.
\end{proof}

Combining these results produces our final main result.

\begin{thm}[Fellow traveling in periodic flats]\label{thm:ft-periodic}
  If $F$ is a potentially periodic triangle-square flat, then there is
  a value $k$ such that pairs of Gersten-Short geodesics in $F$ that
  start and end within $1$ unit synchronously $k$-fellow travel.
\end{thm}

\begin{proof} 
  Since this result has already been established from pure and striped
  flats (Propositions~\ref{prop:ft-pure} and~\ref{prop:ft-striped}),
  we only need to consider flats that are thoroughly crumpled
  (Proposition~\ref{prop:aperiodic-flats}).  So assume that $F$ is a
  potentially periodic thoroughly crumpled triangle-square flat that
  has been colored with corresponding embedding $p$ into $K = \E
  \times \E$.  Let $\alpha$ and $\beta$ are Gersten-Short geodesics in
  $F$ that start and end within $1$ unit of each other.  We show that
  they synchronously $k$-fellow travel by showing that they both
  approximate a pair line segments in the same two directions.  By
  Lemma~\ref{lem:quasi-linear}, so long as the corresponding intervals
  in $K$ are not degenerate parallelopipeds, the images of $\alpha$
  and $\beta$ approximate straight lines in the same direction.
  Moreover, once the intervals do degenerate, the remaining portion
  must approximate the intersection of the affine plane $A$
  approximating $p(F)$ (Proposition~\ref{prop:quasi-flat}) and the
  affine space containing the degenerate parallelopiped.  Because the
  projection maps from thoroughly crumpled planes are onto
  (Corollary~\ref{cor:project-onto}), the intersection of these two
  affine spaces is at most one-dimensional.  Combining the constant
  governing how closely Gersten-Short geodesics travel along a
  straight line in the non-degenerate portion with the constant
  governing how closely $p(F)$ approximates $A$, yields a global
  constant $k$ that is, by construction, independent of $\alpha$ and
  $\beta$ and only dependent on the periodic structure of $F$ itself.
\end{proof}

\section{Establishing biautomaticity}\label{sec:biaut}

In this final section we outline a sequence of steps which, if carried
out, would fully establish that a compact nonpositively curved
triangle-square complex $K$ has a biautomatic fundamental group
(Conjecture~\ref{conj:biaut}).  According to Theorem~\ref{thm:biaut},
in order to establish the biautomaticity of $\pi_1(K)$, one needs to
find a collection of paths and then to prove that they fellow travel
and form a regular language.  We firmly believe that the Gersten-Short
geodesics defined have satisfy these conditions.  Next, since
triangle-square complexes with no immersed flats are already known to
be biautomatic (Theorem~\ref{thm:flats-biaut}), it is natural to focus
attention on the flats that do exist.  A good first step would be to
establish Conjecture~\ref{conj:aperiodic} in order to rule out the
existence of intrinsically aperiodic flat planes focusing attention on
those potentially periodic flats that immerse into $K$.  For each
individual potentially periodic flat plane we have shown that there is
a value $k$, depending on the flat, so the Gersten-Short geodesics
contained in the flat that start and end within $1$ unit of each other
synchronously $k$-fellow travel (Theorem~\ref{thm:ft-periodic}).  The
natural next step would be to prove the following:

\begin{conj}[Thoroughly crumpled planes]
  For any compact nonpositively curved triangle-square complex $K$
  there are only finitely many distinct thoroughly crumpled planes
  that immerse into $K$. As a consequence, there is a global value $k$
  depending only on $K$ so that Gersten-Short paths in any flat inside
  $\widetilde K$ that start and end within $1$ unit of each other
  synchronously $k$-fellow travel.
\end{conj}

Once this point has been reached, one should use the fellow traveling
constants within planes and the $\cat(0)$ nature of $\widetilde K$
overall to prove that all Gersten-Short paths in $\widetilde K$
synchronously $k$-fellow travel for some uniform global $k$.  Since
the constants in pure square, pure triangle and striped flats are
small, and periodic crumpled flats have bounded intersection with
other flats, this seems reasonable.  The necessary arguments might be
similar to those used by Chris Hruska in \cite{Hr05}.  Finally, once
it is known that these paths fellow travel in $\widetilde K$, the
falsification by fellow traveler property \cite{NeSh95} should be
invoked to establish the regularity of the language they describe.

\bibliographystyle{plain}

\end{document}